\documentclass{article}
\usepackage[utf8]{inputenc} % Specifies TeX to process the symbols you type as unicode symbols ("utf8" stands for "Unicode Transformation Format 8")
\usepackage{amsmath} % Provides many features for documents containing mathematical formulas
\usepackage{amsthm} % Helps to define theorem-like structures
\usepackage[upint]{stix2} % Changes font to the stix2 fonts, which provides extended symbols
\usepackage{inconsolata} % Better typewriter font
\usepackage[margin=1in]{geometry}
\usepackage{mathtools}
\usepackage{enumerate} % Allows for a lot of options with lists
\usepackage{xcolor}
\usepackage[colorlinks,allcolors=blue]{hyperref} % You can change the colors of various parts of this package, but here I change all of them
\usepackage[nameinlink]{cleveref} % For easy internal referencing
\usepackage[url=false,eprint=false,style=alphabetic,maxbibnames=99]{biblatex} % For including bibliographies

\usepackage[final]{microtype}
\usepackage{todonotes} % For including comments that appear in the pdf

\swapnumbers % Makes numbers appear before environment names

\usepackage{bookmark}
\bookmarksetup{
  numbered, 
  open,
}

\usepackage{tikz-cd}

\usepackage{verbatim}
\numberwithin{equation}{subsection} % Numbers equations within the section. E.g., the n-th equation of Section s is numbered (s.n)
\newtheorem{theorem}[equation]{Theorem} % This is where you add environment labels. The first brackets contain what you will put in the environment code. The second label is what will be outputted in the PDF.
\newtheorem{lemma}[equation]{Lemma}
\newtheorem{corollary}[equation]{Corollary}

\newtheorem{proposition}[equation]{Proposition}

\theoremstyle{definition} % Makes text in the environment upright
\newtheorem{definition}[equation]{Definition}
\newtheorem{construction}[equation]{Construction}

\newtheorem{example}[equation]{Example}

\newtheorem{remark}[equation]{Remark}

\newtheorem{notation}[equation]{Notation}

\renewcommand{\in}{\smallin}

\newcommand{\categ}{\mathsf}
\newcommand{\Rings}{\categ{Rings}}

\newcommand{\E}[1]{\mathbb{E}_{#1}}

\DeclareMathOperator{\st}{st}

\newcommand{\Prl}{\categ{Pr^L}}
\newcommand{\V}{\mathcal{V}}
\newcommand{\PrlX}[1]{\categ{Pr}^L_{#1}}
\newcommand{\PrlV}{\PrlX{\V}}
\newcommand{\Prlst}{\categ{Pr}^{L, \st}}

\newcommand{\PrlstX}[1]{\categ{Pr}^{L,\st}_{#1}}
\newcommand{\PrlstV}{\PrlstX{\V}}
\DeclareMathOperator{\St}{St}
\DeclareMathOperator{\CoSt}{CoSt}
\DeclareMathOperator{\Ab}{Ab}
\DeclareMathOperator{\Indecomp}{Idcmp}
\DeclareMathOperator{\Prim}{Prim}
\DeclareMathOperator{\Day}{Day}
\DeclareMathOperator{\LMod}{LMod}
\DeclareMathOperator{\LComod}{LComod}
\DeclareMathOperator{\Comod}{Comod}

\newcommand{\bignPrlk}[1]{\categ{Pr^{L, {#1} \kern 1pt \wedge}_k}}
\DeclareMathOperator{\nonunit}{nu}

\DeclareMathOperator{\SymMonCat}{SMonCat}
\DeclareMathOperator{\Exc}{Exc}
\DeclareMathOperator{\Fun}{Fun}

\DeclareMathOperator{\op}{op}
\DeclareMathOperator{\Mod}{Mod}

\DeclareMathOperator{\Spc}{\categ{Spc}}
\DeclareMathOperator{\Cat}{\categ{Cat}}
\DeclareMathOperator*{\colim}{colim}

\DeclareMathOperator{\Hom}{\mathcal{H}\kern -2pt om}

\DeclareMathOperator{\Alg}{Alg}

\DeclareMathOperator{\Coalg}{Coalg}
\DeclareMathOperator{\Op}{Op}
\DeclareMathOperator{\CoOp}{CoOp}

\DeclareMathOperator{\aug}{aug}
\DeclareMathOperator{\End}{End}

\newcommand{\eqTriangle}{{\Delta^2_{=}}}

\newcommand{\A}{\categ{A}}

\newcommand{\C}{\categ{C}}
\newcommand{\D}{\categ{D}}
\newcommand{\M}{\categ{M}}

\DeclareMathOperator{\LM}{LM}

\DeclareMathOperator{\Assoc}{Assoc}

\DeclareMathOperator{\Comm}{Comm}

\DeclareMathOperator{\Sseq}{Sseq}

\DeclareMathOperator{\Triv}{Triv}

\DeclareMathOperator{\CAlg}{CAlg}

\DeclareMathOperator{\Free}{Free}
\DeclareMathOperator{\Cofree}{Cofr}
\DeclareMathOperator{\Tot}{Tot}

\DeclareMathOperator{\Ind}{Ind}

\DeclareMathOperator{\Set}{\mathbf{Set}}
\DeclareMathOperator{\res}{res}

\DeclareMathOperator{\cofib}{cofib}
\DeclareMathOperator{\Lin}{Lin}
\DeclareMathOperator{\Mult}{Mult}

\DeclareMathOperator{\fib}{Fib}

\DeclareMathOperator{\Fin}{Fin}
\DeclareMathOperator{\bij}{bij}
\DeclareMathOperator{\Md}{\mathcal{M}}
\newcommand{\freesym}[1]{\Fin^{\bij}_{#1}}
\newcommand{\freesymtrunc}[2]{\Fin^{\bij}_{#1, \leq {#2}}}
\newcommand{\freesymcoprod}[1]{\Fin^{\bij, \amalg}_{#1}}

\title{Costability of comodules}
\author{Fei Yu Chen}
\date{April 10, 2024}

\begin{document}

\maketitle

\begin{abstract}
    Given a ring $R$, we have a classical result stating that the 
ordinary category of modules $\Mod_R$ is the abelianization of the category of $R$-algebras $\Alg_R$. Using the framework of infinity categories and higher algebra, Francis has shown that the infinity category of $R$-modules is the {\em stabilization} of the infinity category of $R$-algebras \cite{FrancisTangent}. In this work we show a dual result for coalgebras over a cooperad $Q$: namely that given a coalgebra $C$, comodules over $C$ are the costabilization of coalgebras over $C$, which is the universal stable category with a finite colimit preserving functor to coalgebras over $C$.
\end{abstract}

\tableofcontents

\section{Introduction}
Operads have held a strong spot in algebraic theories, especially in higher category theory and homotopical mathematics, for example see \cite{HA, FrancisTangent, Amabel, Hinich:operad, LV}.
There are many ways to connect topological/geometric problems to algebraic ones by taking functions, or if one is more homotopically minded, cohomology cochains. Operads were introduced by May as an way of encoding algebraic laws into a single algebraic gadget \cite{May}. He began with the study of iterated loop spaces and their algebraic structure of multiple composition structures (ie composing along the various dimensions of the $n$-fold loop spaces) and created the $\E{n}$ operad along with $\E{n}$-algebras. 

The dual story of coalgebras is equally interesting. Coalgebras are used in computer science as a way of modeling data structures and computation states, for example see \cite{Jacobs2016IntroductionTC}. They also model topological or geometric data more naturally, as for example every topological space $X$ is a coalgebra for the natural cartesian product functor. Furthermore, even if one only cares about algebras, coalgebraic techniques arise for constructing homotopical resolutions via a bar/cobar duality \cite{LV}. Further, they are used in studying descent data in algebraic geometry, typically arising from a map of schemes $f: X \to Y$ and the induced adjunction pair $f^\ast\dashv f_\ast$. 

In this paper we aim to dualize a well known result about modules over an algebra. Given an ordinary ring $R$ and its $1$-category of algebras $\Alg_R$, we have the following statement
\[
\Mod_R \simeq \Ab(\Alg_R)
\]
where $\Ab(\Alg_R)$ denotes the subcategory of abelian group objects in $\Alg_R$, also called the abelianization of $\Alg_R$. 

This has an infinity categorical version where one uses the notion of stabilization (for example see \cite[Chapter~1]{HA}) instead of abelianization. Suppose we're given an unital operad $P$ and a presentably stable infinity category $\C$. The stabilization is defined as follows:
\begin{definition}[Stabilization]
    Given an infinity category $\C$ with finite limits. Then the stabilization $\St(\C)$ is the universal stable category with a finite limit preserving functor to $\C$. In other words, given any stable infinity category $\D$ and finite limit preserving functor $f: \D \to \C$, we have a universal factorization:
    \[
    \begin{tikzcd}
        \D \ar[d, dotted] \ar[r, "f"] & \C \\
        \St(\C). \ar[ru, "\Omega^{\infty}"']
    \end{tikzcd}
    \]
\end{definition}

Then given a $P$-algebra $R \in \Alg_P(\C)$, Francis proved that
\[
\St (\Alg_{P}(\C)^{/R}) \simeq \Mod^P_R(\C)
\]
where $\Mod^P_R(\C)$ denotes the infinity category of $P$-operadic modules over $R$. Specializing to the case of $P = \Comm$ the commutative operad or $P = \Assoc$ the associative operad gives us the usual result that the stabilization of commutative (resp. associative) algebras over $R$ is equivalent to commutative (resp. associative) modules over $R$. Notice that an ``associative module'' here refers to a bimodule.

In this paper, we prove the dual statement that comodules are the costabilization (see \cite[Section~5]{FrancisThesis}) of coalgebras. Costabilization can be characterized as follows:
\begin{definition}[Costabilization]
    Given an infinity category $\C$ with finite colimits. Then the costabilization $\CoSt(\C)$ is the universal stable infinity category with a finite colimit preserving functor to $\C$. In other words, given any stable infinity category $\D$ and finite colimit preserving functor $f: \D \to \C$, we have a universal factorization:
    \[
    \begin{tikzcd}
        \D \ar[d, dotted] \ar[r, "f"] & \C \\
        \CoSt(\C). \ar[ru, "\Sigma_{\infty}"']
    \end{tikzcd}
    \]
\end{definition} Compare with stabilization where the condition is on finite {\em limit} preserving functors from stable categories instead of finite {\em colimit} preserving ones.

Now a key problem in trying to dualize directly by just taking opposites is the following: the opposite of a presentable infinity category is no longer presentable. Relatedly, the tensor product of a presentably symmetric monoidal infinity category $\C$ only commutes with colimits and not limits. This prevents many techniques for algebras from working with coalgebras; namely one has to be much more careful when taking cofree/cobar resolutions. Given an operad $P$ and a $P$-algebra $A$, the free/bar resolution
% https://q.uiver.app/#q=WzAsNCxbMCwwLCJBIl0sWzEsMCwiUCBcXGNpcmMgQSJdLFsyLDAsIlAgXFxjaXJjIFAgXFxjaXJjIEEiXSxbMywwLCJcXGRvdHMiXSxbMSwwXSxbMiwxLCIiLDAseyJvZmZzZXQiOjF9XSxbMiwxLCIiLDAseyJvZmZzZXQiOi0xfV1d
\[\begin{tikzcd}
	A & {P \circ A} & {P \circ P \circ A} & \dots
	\arrow[from=1-2, to=1-1]
	\arrow[shift right, from=1-3, to=1-2]
	\arrow[shift left, from=1-3, to=1-2]
\end{tikzcd}\]
is a simplicial resolution, thus a  colimit, thus commutes nicely with the tensor. With coalgebras, the cobar resolution is a totalization, so a limit, hence it doesn't work well with the tensor product.

Another problem with the coalgebra case is if one tries to follow the Francis' proof \cite{FrancisTangent}, he requires many techniques that rely on theories of Boardman-Vogt tensor of operads and on Goodwillie derivatives of split analytic functors. The definition of Goodwillie derivatives unfortunately gives adjunctions ``in the wrong way'', which give monads and algebras instead of comonads, which make it difficult to apply comonadicity theorems. The universal property is as follows (see \cite{Goodwillie, HA}): given infinity categories $\C, \D$ where $\C$ has finite colimits and $\D$ has finite limits, we have an inclusion 
\[
\Exc(\C, \D) \to \Fun(\C, \D)
\]
where excisive functors send pushouts to pullbacks. The Goodwillie derivative $DF: \C\to \D$ of a functor $F: \C \to \D$ is a local left adjoint to this inclusion. In other words, natural transformations $F \to G$ where $G$ is excisive factors through $F \to DF \to G$ where $F \to DF$ is the unit morphism. Now notice that on $\Fun(\C, \D)$, we have a monad instead of a comonad. This is however likely rectifiable, as one can define Goodwillie ``co-derivatives'' with a dual universal property, though one must set up the theory.

A deeper problem is the Boardman-Vogt tensor product. This is a special tensor operation on topological operads (or set-based operads), which relies on the fact that topological spaces have a diagonal morphism $X \to X \times X$. For example, even $1$-categorically for set based operads, the Boardman-Vogt tensor operation requires the iterated diagonal morphism to copy the same operation many times, see for example \cite[Section~5.1]{MoerdijkWeiss}. We thus cannot dualize this at least a priori, since spaces do not have a codiagonal map $X \times X \to X$.

We avert the use of these techniques by using an alternative argument. We work with conilpotent coalgebras with divided powers, following \cite[Chapter~4]{DAG}.

Our main theorem is the following:
\begin{theorem}\label{theorem:main}
    Fix a characteristic zero field $k$.
    Given a presentably stable $k$-linear category $\V$ and a conilpotent cooperad $Q$, as well as a coalgebra $C \in \Coalg_Q(\V)$. Then we have an equivalence
    \[
    \CoSt(\Coalg_Q^{C/}(\V)) \simeq \Comod_C^Q(\V).
    \]
\end{theorem} Notice that costabilization satisfies $\CoSt(\C^{/c}) \simeq \CoSt(\C)$, so we will instead focus on $\CoSt(\Coalg_Q^{C/-/C})$ in the proof, see \ref{subsec:costcomod}.

The paper is structured as follows: we begin with a recollection about symmetric sequences and (co)operads, as well as (co)algebras and (co)modules over them.

Then we move on to discuss the tangent complex construction (\ref{constr:tangent}) for coalgebras, a dual version of the usual cotangent complex of algebras. This functor is strongly related to the eventual comparison functor $\CoSt(\Coalg_Q^{C/-/C}(\V)) \to \Comod^Q_C(\V)$.

Finally we move on to the proof of the main theorem. This requires an understanding of pushouts of coalgebras (\ref{prop:coalg_pushout}) as well as an understanding of the tangent complex construction as well as its left adjoint. 
\subsection{Notation}
\begin{enumerate}
    \item $A, B$ are typically sets. These will typically be the set of colors of an operad/cooperad.

    \item $\C, \D$ are categories. By ``category'', we by default mean $(\infty, 1)$-category (also sometimes just called $\infty$-category). We will mention ordinary $1$-categories explicitly. 

    \item $\V$ is a presentably symmetric monoidal category, later on assumed also to be stable.  

    \item An $A$-labelled set is a set $[n] \to A$. We usually denote it as $\{a_1, \dots, a_n\}$ where repetitions are allowed. This is a "multiset" of $A$. 
    \item $\C^{\ast}$: add zero object. If $\C$ is already symmetric monoidal, can extend the product by letting tensoring with $0$ be canonically $0$. 

    \item $\Prl, \Prlst$: The category of small presentable categories. Similarly $\Prlst$ denotes the category of small presentable stable categories. These are subject to size conditions, as discussed in the next section.
    
    \item $\PrlV, \PrlstV$: Given a presentably symmetric monoidal category $\V$, we will use $\PrlV$ to denote the category of presentable categories with a presentable $\V$-action. Similarly if $\V$ is stable as well, then we let $\PrlstV$ denote the category of presentable stable categories with a presentable $\V$-action.  
    
\end{enumerate}
\subsection{Trees and (co)operads}

We also briefly use the language of trees to discuss operadic and cooperadic constructions. We follow a definition akin to  \cite[Definition~3.2.3]{Hinich:operad}. Our trees are nonplanar, has one initial/root vertex, and finitely many terminal/leaf vertices, and a nonzero number of internal vertices. Our trees can also be marked by elements of a symmetric sequence $V$: each internal vertex $x$ is then marked by some $v \in V(val(x))$ based on how many incoming edges there are of $x$. 

When we discuss modules (or comodules), the terminal edges of a tree are also marked by $a$ or $m$, purely decorators to indicate which edges should take the algebra as input and which should take the module. Here is an example of a 3-leaf corolla which is marked by $v \in V(3)$.
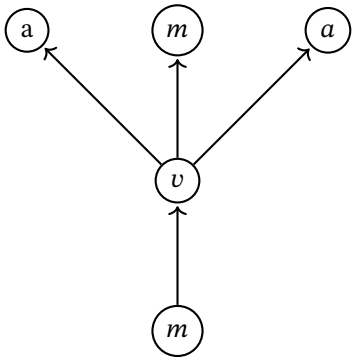
\[
\begin{tikzpicture}[->,shorten >=1pt,auto,node distance=2cm,
                    thick,main node/.style={circle,draw}]

  \node[main node] (1) {a};
  \node[main node] (2) [right of=1] {$m$};
  \node[main node] (3) [right of=2] {$a$};
  \node[main node] (4) [below of=2] {$v$};
  \node[main node] (5) [below of=4] {$m$};

  \path[every node/.style={font=\sffamily\small}]
    (4) edge node {} (1)
    (4) edge node {} (2)
    (4) edge node {} (3)
    (5) edge node {} (4);
\end{tikzpicture}
\]

\subsection{Tensored categories and (co)modules}
We discuss ways to give maps of (co)modules. This helps us define various functors of operads and cooperads because they are special cases of the data of an action category (which will symmetric sequences) acting on a module category (which will be the underlying category we're taking algebras/coalgebras on). Here we use the notion of tensored categories as developed in \cite[Chapter~4.2]{HA}. Now we begin with a definition of lax and colax functors of $\A$-tensored categories. 

\begin{definition}[Lax and colax morphisms]\label{def:lax}
Let $\C$ and $\D$ be tensored over a category $\A$. This is equivalent to $(\A, \C)$ and $(\A, \D)$ having $\LM$-monoidal operadic structures $\C^{\LM}, \D^{\LM}$. Then we define a {\em lax $\A$-monoidal functor} to the be structure of an $\LM$-operadic morphism $\C \to \D$ as $\LM$-operads, which is identity on the $\A$-factors. 

Similarly, a {\em colax $\A$-monoidal functor} is the structure of an $\LM$-operadic morphism $\C^{\op} \to \D^{\op}$ as $\LM$-operads which is identity on the $\A^{\op}$-factors. 

The data of a (co)lax morphism is usually by abuse of notation just referred to by the underlying functor on categories $\C \to \D$. 

As a variant of the above concepts, suppose we're given $\C, \D$ both $\A$-tensored categories. Then a {\em strongly $\A$-monoidal functor} from $\C \to \D$ is a morphism of $\LM$-operads $\C \to \D$ which is identity on the $\A$-factors and which preserves coCartesian lifts of active maps in $\LM$. Equivalently, it is a coCartesian fibration.  
\end{definition}

\begin{remark}
    Notice here we're using the fact that given an $\A$-tensored category $\C$ (with structure given by $\C^{\LM} \to \LM$), then $\C^{\op}$ is naturally $\A^{\op}$-tensored. Here we're using the natural structure of an $\LM$-monoidal structure on $\C^{\op}$ given the ones on $\C$. Note this structures is {\em not} just the opposite of the category $\C^{\LM}$. Instead one must take a fiberwise opposite construction of the coCartesian fibration $\C^{\LM} \to \LM$, as given by Barwick, Glasman and Nardin in \cite{BarwickGlasmanNardin}. If one thinks of this coCartesian fibration as a functor $\LM \to \Cat$, then this fiberwise opposite is composing this functor with the natural opposite involution $\op: \Cat \to \Cat$.

    Using Barwick, Glasman, and Nardin's construction \cite{BarwickGlasmanNardin}, we see the notion of strongly monoidal $\A$-functor is self dual, in other words it is both a lax and colax morphism.
\end{remark}

Then following \cite[Chapter~4.2]{HA}, we note that a lax $\A$-monoidal morphism $\C \to \D$ gives rise to a functor of left modules $\Mod_A(\C) \to \Mod_A(\D)$ for any $A \in \Alg(\A)$. Dually, a colax $\A$-monoidal morphism $\C \to \D$ gives rise to a functor of left comodules $\Comod_C(\C) \to \Comod_C(\D)$.

%Further discuss how we manifest colax morphisms of categories. 

We for the most part with discuss when a functor $f: \C \to \D$ have a (co)lax structures on them. To do so, we usually construct the first step of the structure, ie a natural transformation from $f(c \otimes c') \to f(c) \otimes f(c')$ (for the lax case) and leave the rest of the structure implied. For the strongly monoidal case, this morphism would be a natural equivalence. The coherence data is usually clear from context by construction: for example if the lax morphism is some inclusion of a direct summand into a direct sum. 

We end with a brief lemma on right adjoints to strongly monoidal functors.
\begin{lemma}
    Given a strongly $\A$-monoidal functor $\pi: \C \to \D$ for $\A$-tensored categories $\C, \D$. Say $\pi: \C \to \D$ has a right adjoint $R$ (which is just a functor from $\D \to \C$). Then $R$ has a natural lax $\A$-monoidal structure. Furthermore, the induced functor $\pi: \LMod_A(\C) \to \LMod_A(\D)$ has a right adjoint given by the one induced by $R$.
\end{lemma}

\begin{proof}
    We sketch the argument. We just need to extend the map from $\D \to \C$ to an operadic map on $\D^{\LM} \to \C^{\LM}$. To do so, we can focus only on the active morphisms, as the inert morphisms have a canonical lifting from the assignment $a_1, \dots, a_n, d \mapsto a_1, \dots, a_n, Rd$.

    Now we have active morphisms of two kinds. The first kind are those that arise from $\A^n \to \A$, which is trivial since we must have $R$ be identity on the multiplication structure of $\A$. The second kind are those that arise from $\A^n\times \D \to \D$. Such an active morphism arises from the tensor structure on $\D$, and can be seen as a morphism 
    \[
    f: (a_1, \dots, a_n, d) \to d',
    \]
    or equivalently a morphism
    \[
    f: (a_1\otimes \dots \otimes a_n \otimes d) \to d'.
    \]
    
    Such morphisms must map to the morphism 
    \[
    (a_1, \dots, a_n, Rd) \to R(a_1 \otimes \dots \otimes a_n \otimes d) \to Rd' 
    \]
    where the second morphism is natural and arises from $R$ applied to $f$. The first morphism arises from the adjunct of 
    \[
    \pi(a_1 \otimes \dots \otimes a_n \otimes Rd) \simeq a_1 \otimes \dots \otimes a_n \otimes \pi Rd \to a_1 \otimes \dots \otimes a_n \otimes d
    \]
    where the first equvalence is from the strong monoidality of $\pi$, and the second is the counit of $d$. This construction is easily seen to be natural in the active morphism $f$, hence extends to a functor on all of the active morphisms of $\D^{\LM}$.

    The second part of the lemma about morphisms of modules is obvious from our constructions.
\end{proof}

\subsection{Sizes of categories, presentability, and foundational considerations}\label{sec:indcomp}
%German's convention: can assume the base category is always small. It will always be U-small for some universe U. 

Consider the category $\Mod_k$ of $k$-modules. Since any given $k$-module can have any cardinality, a priori we see that $\Mod_k$ cannot have just a set of objects, but a formal class of objects. 

This is also true for various types of categories defined as structured sets: they must have a formal class of objects already.

However, next consider $\Cat$, the category of categories. Intuitively its objects should consist of categories like $\Mod^{\op}_k$ and other structured sets. This gives us a dilemma: the objects of $\Cat$ are too large to even form a formal class. Similar problems exist when pondering related categories like $\Prl$ of presentable categories, which require a notion of size even to define. This problem also occurs when considering various forms of completions of categories, like $\Ind$-completion or colimit completion. For example, if we take $\Ind(\Mod^{\op}_k)$, the completion of $\Mod_k$ under small (here meaning set-sized) filtered colimits, the result will again be too large to even have a class of objects.

Our approach is to hypothesize a sequence of Grothendieck universes. More precisely, we can define a Grothendieck universe $U$ as a set closed under $\in$, power set, and $U$-small unions. Such choosing such a set $U$ is equivalent to choosing an inaccessible cardinak $\kappa$, as all Grothendieck universes are of the form $U = V_\kappa$ for some inaccessible cardinal $\kappa$.

We choose a sequence 
\[
U_0 \subseteq U_1 \subseteq U_2 \dots
\]

This approach is very versatile: now we can generally define $U_i$-small structured sets. The category of all $U_i$-small $k$-modules, $\Mod_{k,i}$, is then $U_{i+1}$-small. Then the category of all $U_{i+1}$-small categories $\Cat_{i+1}$ is $U_{i+2}$-small. This approach lets us take even further $\Ind$-completions: for example $\Ind_{i+2}(\Cat_{i+1})$ (adding all $U_{i+2}$-small filtered colimits) is $U_{i+3}$-small. 

However for the most part, we will just need three layers: a ``small'' (which we choose to be uncountable at least) universe $U$, a ``large'' universe $V$, and a ``very large'' universe $W$. Then again we have the three layers of sizes:
\begin{enumerate}
    \item $U$-small, which we also call ``small''. Our structured sets like $k$-modules, $k$-algebras, etc. will be small.
    \item $V$-small, which we also call ``large''. Our categories $\Mod_k$, $\Alg_k$ are here.
    \item $W$-small, which we also call ``very large''. Now the various categorical constructions like $\Cat$, and $\Prl_U$ lie here. 
\end{enumerate}
Very rarely we may require another universe above $W$. The theory of Grothendieck universes make it easy to hypothesize another $W'$ above $W$ if necessary.

Notice that the notions of limit/colimit completeness as well as presentability depend on sizes. We henceforth introduce a notation for presentability and completeness.

\begin{notation}[Presentability]
If $\C$ is a $U_{i+1}$ sized category, we say that it is presentable to mean that it is presentable using $U_i$-colimits, or $U_i$-presentable. This is similar to the size bounds needed to discuss completeness and cocompleteness. These notions of presentability and completeness are usually left implicit, and is clear depending on the category we're talking about. Thus if $\D$ is a large category which we call presentable, it by default means that $\D$ is presentable using small colimits. 

Notice that if $\C$ is a $U$-small category which is $U$-presentable, then by familiar arguments it has all $U$-colimits, and thus must be a poset. Since we do not wish to study such degenerate cases, we thus picked the above convention, where by default the presentability of $\C$ means that it uses colimits one size smaller instead. 

Notice also that if $\C$ is $U_0$-small, which means finite, then there is no way that $\C$ is presentable unless $\C$ is a poset (and has all finite colimits). Hence we also leave implicit that when we say $\C$ is presentable, by default we do not allow $\C$ to be finite.

As an abuse of notation, regardless of the size of $\C$, we write $\C \in \Prl$ to mean that $\C$ is presentable. If $\C$ is large, this $\Prl$ would be $\Prl_U$, the category of presentable categories under $U$-small colimits. If $\C$ is very large, the $\C \in \Prl$ would mean that $\C$ is presentable using $V$-small colimits. Very rarely, if $\C$ is very very large, ie $W'$-small where $W'$ is the universe above $W$, then $\C \in \Prl$ would mean that $\C$ is presentable using $W$-colimits.
\end{notation}

\subsection{Acknowledgements}

I'd like to thank my advisor, David Nadler, from whose guidance I have benefited enormously. In addition, I'd like to thank Germ\'an Stefanich, with whom I've had many interesting conversations and discussions on higher category theory and higher algebra. I'd also like to thank John Francis for several useful chats on the stabilization of algebras and bar/cobar duality.

\section{Symmetric sequences, (co)operads}
We define symmetric sequences and (co)operads. We follow the standard method of defining a circle product on symmetric sequences via a universal property. However for us we need to define multi-object symmetric sequences, in order to define multi-colored (co)operads, see \cite{LV}, \cite{Amabel}, \cite[Chapter~6]{DAG}, \cite{Ching}.

\subsection{Multi-object symmetric sequences}
We begin by defining a multi-object version of $\Fin^{\bij}$, the $1$-category of finite sets and bijections.

\begin{definition}[$\freesym{A}$]
Let $A$ be a set. We let $\freesymcoprod{A}$ denote the $1$-category whose objects are given by finite sets labelled with elements of $A$. In other words, they are given by maps $f: X \to A$ where $X$ is a finite set. These are all isomorphic to some $[n] \to A$, which can be thought of as a "multiset" with elements in $A$. They shall be denoted by $\{a_1, \dots, a_n\}$ where the $a_i$ are allowed to repeat.

The morphisms are given by bijections $\sigma$ of $X$ with itself, which takes $f: X \to A$ to 
\[
\begin{tikzcd}
    X \ar[r, "\sigma^{-1}"] & X \ar[r, "f"] & A.
\end{tikzcd}
\]
The monoidal product is given by disjoint union. Note we sometimes drop the $\amalg$ and write $\freesym{A}$ instead.
\end{definition}
\begin{remark}\label{rmk:freesym}
Note that $\freesymcoprod{A}$ is the free symmetric monoidal category generated by $A$, even infinity categorically. The monoidal product is given by taking disjoint unions, ie coproducts. 
\end{remark}

Next we mimic the definition of symmetric sequences over $\V$, which is usually defined as $\Fun(\Fin^{\bij}, \V)$, where $\V$ is say a presentably symmetric monoidal category, ie $\V \in \CAlg{\Prl}$. The intuition is that each cardinality $n$ gets sent to an object $S(n)$ which would encode the $n$-ary operations.

Here we want operations for each domain $\{a_1, \dots, a_n\}$ and codomain $a$. Here the domain is given by a finite set labelled by $A$ (so a map $X \to A$) and the codomain is a single element of $A$. To capture this data, we take functors out of $A \times \freesym{A}$, where the first component gives the codomain and the second component gives the domain.

\begin{definition}[$\Sseq_A$]\label{def:sseq}
    Given a set $A$ and a presentably symmetric monoidal category $\V$, we define {\it symmetric sequences over $A$ in $\V$} as follows:
    \[
    \Sseq_A(\V) := \Fun(\freesym{A} \times A, \V).
    \]
    Note that we may only write $\Sseq_A$ when $\V$ is clear from context.
\end{definition}

We also have a following variation of truncatd symmetric sequences using the truncated category $\freesymtrunc{A}{n}$, whose objects consists of sets with cardinality $\leq n$, and is otherwise the same as $\freesym{A}$.
\begin{definition}
Given $A$ a set and $\V \in \CAlg(\Prl)$.

We can define {\em truncated symmetric sequences over $A$ in $\V$} via 
\[
\Sseq_{A,\leq n}(\V) := \Fun(\freesymtrunc{A}{n} \times A, \V).
\]
Note that we may only write $\Sseq_{A, \leq n}$ when $\V$ is clear from context. 
\end{definition}

\subsection{Circle product}
Next we move on to defining the circle product. We follow an argument that uses universal properties, again following \cite{LV}, \cite{Amabel}, \cite[Chapter~6]{DAG}, \cite{Ching}.

\begin{lemma}\label{lemma:dayconequiv}
Let $A$ be a set and $\V^\otimes$ be a presentably symmetric monoidal category. Then $\Fun(\freesymcoprod{A}, \V^\otimes)$, with its Day convolution product, is the free $\V$-tensored presentably symmetric monoidal category. In other words, given $\C^\otimes \in \CAlg(\PrlV)$, the natural evaluation map
\[
\Fun_{\CAlg(\PrlV)}(\Fun(\freesym{A}, \V)^{\otimes_{\Day}}, \C^\otimes) \to \Fun(A, \C)
\]
is an equivalence. 
\end{lemma}

\begin{proof}
Here the evaluation map is described as follows: we use the inclusion $A \to \freesym{A}$ taking $a \in A$ to singleton labelled sets. There is a Yoneda embedding $y: \freesym{A} \to \Fun(\freesym{A}, \Spc)$ (note here that $\freesym{A}$ is equivalent to its opposite canonically; it's a groupoid). Lastly, we use $\Fun(\freesym{A}, \Spc) \otimes \V \simeq \Fun(\freesym{A}, \V)$ as presentably symmetric monoidal categories. Then the evaluation is given by
\[
\begin{aligned}
\Fun_{\CAlg(\PrlV)}(\Fun(\freesym{A}, \V), \C) \simeq \Fun_{\CAlg(\Prl)}(\Fun(\freesym{A}, \Spc), \C) \simeq \Fun_{\SymMonCat}(\freesym{A}, \C) \simeq \Fun(A, \C).
\end{aligned}
\]
The first equivalence is via the adjunction between $\V \otimes$ and the forgetful functor. The second is using the free completeness property of $\Fun(\freesym(A), \C)$, and the third is precomposition by $y$, which is an equivalence because $\freesym{A}$ is the free symmetric monoidal category generated by $A$ (\ref{rmk:freesym}). 
\end{proof}

Now we can use (\ref{lemma:dayconequiv}) by plugging in $\C = \Fun(\freesym{A}, \V)^{\otimes_{\Day}}$ and we get the following equivalence:
\begin{align*}
\Fun_{\CAlg(\PrlV)}(\Fun(\freesym{A}, \V), \Fun(\freesym{A}, \V)) & \simeq \Fun(A, \Fun(\freesym{A}, \V)) \\
& \simeq \Fun(A \times \freesym{A}, \V) =: \Sseq_A(\V)
\end{align*}
Now note that the left hand side has a composition, since it's the endomorphisms of $\Fun(\freesym{A}, \V)$, which hence induces a circle product on $\Sseq_A(\V)$.

\begin{definition}[Circle product]\label{def:circ}
Given a set $A$ and a presentably symmetric monoidal category $\V$, we define the circle product $\circ$ on $\Sseq_A(\V)$ to be the opposite of the composition product given by the equation 
\[
\End_{\CAlg(\PrlV)}(\Fun(\freesym{A}, \V)) \simeq \Sseq_A(\V).
\]
\end{definition}
\begin{remark}
Note we reverse the composition product to follow usual conventions.
\end{remark}

\begin{remark}\label{rmk:circformula}
    We can calculate out the circle product in detail: given a symmetric sequences $S \in \Sseq_A(\V)$, we note that in $\End_{\CAlg(\PrlV)}(\Fun(\freesym{A}, \V))$, $S$ corresponds to the functor that takes $X \in \Fun(\freesym{A}, \V)$ to 
    \[
    \colim_{\{a_1, \dots, a_n\} \to X} S(a_1) \otimes \dots \otimes S(a_n) 
    \]
    where $\{a_1, \dots, a_n\} \to X$ ranges through all maps from a representable (perhaps tensored with $\V$) into $X$.

    Then, given another symmetric sequence $T$, we see that $S \circ T(\{a_1, \dots, a_n\}, a)$ should evaluate to 
    \[
    \colim_{\{b_1, \dots, b_m\} \to S(a)} T(\{a_1, \dots, a_{i_1}\}, b_1) \otimes \dots \otimes T(\{a_{i_{m+1}}, \dots, a_n\}, b_m).  
    \]

    Now this is a particularly simple colimit due to the nature of our category: it evaluates to 
    \[
    \bigoplus_{\{b_1, \dots, b_m\}} S(\{b_1, \dots, b_m\}, a) \otimes_{\Sigma_m} [T(\{a_1, \dots, a_{i_1}\}, b_1) \otimes \dots \otimes T(\{a_{i_{m+1}}, \dots, a_n\}, b_m)].  
    \]
    which is the usual circle product formula!
\end{remark}
\begin{remark}\label{rmk:circunit}
The unit for the $\circ$ operation is easily read off from the evaluation map (\ref{def:circ}): it is the $A$-symmetric sequence $\Triv$ such that 
\begin{itemize}
    \item $\Triv(a, a)$ is the unit for the monoidal structure on $\V$
    \item Otherwise everything else evaluates to $0 \in \V$, the initial object. 
\end{itemize}
\end{remark}

\subsection{(Co)Operads and (co)algebras over them}
Now we can define operads and cooperads.

\begin{definition}[Operads and cooperads]
    Let $A$ be a set, and $\V$ be a presentably symmetric monoidal category. Then we define the category of {\it operads in $\V$ over $A$} to be 
    \[
    \Op_A(\V) := \Alg(\Sseq^{\circ}_A).
    \]
    Similarly, we define the category of {\it cooperads in $\V$ over $A$} to be 
    \[
    \CoOp_A(\V) := \Coalg(\Sseq^{\circ}_A).
    \]
    We may drop the $\V$ from the notation when the context is clear. 
\end{definition}

\begin{example}
In any monoidal category, the unit is naturally an algebra and coalgebra. Thus, $\Triv$ (\ref{rmk:circunit}) is canonically an operad and cooperad. 
\end{example}

\begin{remark}\label{remark:univ_monad}
We can see operads also as ``universal monads'' in the following way. Notice that the as monoidal categories, 
\[
\Sseq_A(\V)^{\circ} \simeq \End_{\CAlg(\PrlV)}(\Fun(\freesym{A}, \V)).
\]
By \ref{lemma:dayconequiv} as well as the Yoneda lemma, we notice that 
\[
\End_{\CAlg(\PrlV)}(\Fun(\freesym{A}, \V)) \simeq \End_{\Fun(\CAlg(\PrlV), \Cat)}(\Fun(A, {-})).
\]
So (co)algebras in $\Sseq_A(\V)^{\circ}$ are equivalent to (co)algebras 
\[
\End_{\Fun(\CAlg(\PrlV), \Cat)}(\Fun(A, {-})).
\]

Thus an operad can equivalently be given as a monad on the functor $\C \mapsto \Fun(A, \C)$ for $\C \in \CAlg(\PrlV)$. Intuitively one can say that an operad is ``monad on $\Fun(A, \C)$ which is natural in $\C$''. 

Dually, we have the same result for cooperads: they are comonads for the functor $\C \mapsto \Fun(A, \C)$ for $\C \in \CAlg(\PrlV)$. Intuitively one can say that a cooperad is then a ``comonad on $\Fun(A, \C)$ which is natural in $\C$''. 
\end{remark}

\begin{definition}\label{def:nonunital}
Given a set $A$ and a stably presentably monoidal category $\V$, one can define a variant of operads and cooperads: A (co)operad $S$ is {\it $1$-reduced} if $S(\empty, a)$ is the initial object for all $a \in A$, and $S(a, b)$ is either the unit of $\V$ if $a = b$, or is the inital object if $a \neq b$. 

One can see that such data is captured by the symmetric sequence's values on $A$-labelled sets of cardinality two or larger, and indeed one can define $1$-reduced (co)operads by starting with the category $\freesym{\geq 2, A}$.

In the sequel, we mainly work with $1$-reduced (co)operads. These objects should be thought of as ``augmented'' or ``non-unital'' variants of (co)operads.
\end{definition}
\begin{remark}\label{rmk:conilpcoop}
These cooperads, when they are $1$-reduced, are usually called {\it conilpotent cooperads with divided powers}, see \cite{Amabel}, \cite[Chapter~6]{DAG}.
%We will discuss more general cooperads in detail later... cite sources here1
\end{remark}

\begin{remark}
    When $S$ is a $1$-reduced (co)operad, there's a natural maps 
    \[
    \Triv \to S \to \Triv
    \]
    of (co)operads. One of the maps is the (co)unit and the other is a sort of (co)augmentation map, hence why we should think of $1$-reducedness as a sort of augmentation. %Araminta, Ex 3.12
\end{remark}

We can next define algebras and coalgebras over operads and cooperads as follows: 
\begin{definition}
Given $A$ a set, $\V$ a presentably symmetric monoidal category, and $\C \in \CAlg(\PrlV)$. There is a left action of $\Sseq_A$ on $\Fun(A, \C)$ (also denoted by $\circ$) using the natural action of $\End_{\CAlg(\PrlV)}(\Fun(\freesym{A}, \V))$ on $\Fun_{\CAlg(\PrlV)}(\Fun(\freesym{A}, \V), \C)$ by precomposition (and using lemma (\ref{lemma:dayconequiv})). 

Hence given an operad $P$, we can define the category of {\it algebras over $P$ in $\C$} to be
\[
\Alg_P(\C) := \LMod_P(\Fun(A,\C)).
\]
Similarly, give a cooperad $Q$, we can define the category of {\it coalgebras over $Q$ in $\C$} to be 
\[
\Coalg_Q(\C) := \LComod_Q(\Fun(A, \C)).
\]
\end{definition}

\begin{remark}
One can calculate that the left action of $\Sseq_A$ on $\Fun(A, \C)$ has an explicit formula, just like the circle product does (\ref{rmk:circformula}). It is given exactly the same formula, just letting $S \in \Sseq_A$ and $T \in \Fun(\A, \C)$:
\[
S \circ T \simeq 
    \bigoplus_{\{b_1, \dots, b_m\}} S(\{b_1, \dots, b_m\}, a) \otimes_{\Sigma_m} [T(b_1) \otimes \dots \otimes T(b_m)].  
\]
\end{remark}

We'll fix a notational convention.

\begin{notation}[Convention on $A, \V, \C$]\label{notation:AVC}
   We fix a set $A$, $\V \in \CAlg(\Prl)$, and $\C \in \CAlg(\PrlV)$. (Co)operads will be $A$-colored and over $\V$, and (co)algebras will be taken over $\C$.
\end{notation}
\subsection{Adjunctions of (co)algebras}
%DEFINE ADJUNCTIONS OF ALGEBRAS/COALGRBRAS: free forgetful, triv/Prim

We use the functoriality of modules over algebras (and comodules over coalgebras) \cite[Section~4.2.3]{HA} to define various adjunctions of (co)algebras over (co)operads (see discussion preceeding Example 2.10 and Example 2.14 of \cite{Amabel}). In particular, if we have an operad map $f: P \to P' \in \Op_A$, this induces a natural adjunction
% https://q.uiver.app/#q=WzAsMixbMSwwLCJcXGJ1bGxldCJdLFswLDAsIlxcYnVsbGV0Il0sWzEsMCwiIiwwLHsib2Zmc2V0IjotMn1dLFswLDEsIiIsMCx7Im9mZnNldCI6LTJ9XV0=
\begin{equation}\label{eqn:algadj}
    \begin{tikzcd}
	\Alg_P & \Alg_{P'}
	\arrow["f_{!}"{name=0}, shift left=2, from=1-1, to=1-2]
	\arrow["f^{\ast}"{name=1},shift left=2, from=1-2, to=1-1]
 \arrow["\dashv"{anchor=center, rotate=-90}, draw=none, from=0, to=1]
\end{tikzcd}
\end{equation}
given by "extension by scalars" and "restriction" respectively. Notation from \cite{Amabel}

We apply this to the unit map $\Triv \to P$:
\begin{example}[Free-forgetful adjunction]\label{ex:freeforgetful}
Fix $A \in \Set, \V \in \CAlg(\Prl), \C \in \CAlg(\PrlV)$, following \ref{notation:AVC}. Then
given an operad $P \in \Op_A(\V)$, the {\it free-forgetful adjunction} is gotten from applying \eqref{eqn:algadj} to the unit map $\Triv \to P$, noting that $\Alg_{\Triv}(\C) \simeq \Fun(A, \C)$. This gives us:

\begin{equation}
    \begin{tikzcd}
    \Fun(A, \C) & \Alg_{P}
	\arrow["\Free"{name=0}, shift left=2, from=1-1, to=1-2]
	\arrow["U"{name=1},shift left=2, from=1-2, to=1-1]
    \arrow["\dashv"{anchor=center, rotate=-90}, draw=none, from=0, to=1]
\end{tikzcd}
\end{equation}
\end{example}
When $P$ is augmented (there is a canonical augmentation when $P$ is $1$-reduced), we also have an augmentation map $P \to \Triv$, giving us another adjunction:
\begin{example}[Indecomposables-trivial adjunction]
    Fix $A \in \Set, \V \in \CAlg(\Prl), \C \in \CAlg(\PrlV)$, following \ref{notation:AVC}.
    Suppose we're given an operad $P \in \Op_A(\V)$. Let $P$ have an augmentation $\epsilon: P \to \Triv$. We apply \eqref{eqn:algadj} to $\epsilon$ and get the {\it indecomposables-trivial algebra} adjunction:
    \[
    \begin{tikzcd}
	\Alg_P(\C) & \Fun(A, \C).
	\arrow["\Indecomp"{name=0}, shift left=2, from=1-1, to=1-2]
	\arrow["\Triv"{name=1},shift left=2, from=1-2, to=1-1]
    \arrow["\dashv"{anchor=center, rotate=-90}, draw=none, from=0, to=1]
    \end{tikzcd}
    \]
\end{example}

Similarly, given a map of cooperads $g: Q \to Q' \in \CoOp_A$, we have a natural adjunction

\begin{equation}\label{eqn:coalgadj}
    \begin{tikzcd}
	\Coalg_Q & \Coalg_{Q'}
	\arrow["g^L"{name=0}, shift left=2, from=1-1, to=1-2]
	\arrow["g_R"{name=1},shift left=2, from=1-2, to=1-1]
 \arrow["\dashv"{anchor=center, rotate=-90}, draw=none, from=0, to=1]
\end{tikzcd}
\end{equation}
given by "co-restriction/extension" and "co-extension/restriction of scalars". 

\begin{example}[forgetful-cofree]\label{ex:forgetfulcofree}
Fix $A \in \Set, \V \in \CAlg(\Prl), \C \in \CAlg(\PrlV)$, following \ref{notation:AVC}.
Suppose we're given a cooperad $Q \in \CoOp_A(\V)$. Then the {\it forgetful-cofree adjunction} is gotten from applying \eqref{eqn:coalgadj} to the counit map $Q \to \Triv$. This gives us:

\begin{equation}
    \begin{tikzcd}
	\Coalg_Q & \Fun(A, \C)
    \arrow["U"{name=0}, shift left=2, from=1-1, to=1-2]
	\arrow["\Cofree"{name=1},shift left=2, from=1-2, to=1-1]
 \arrow["\dashv"{anchor=center, rotate=-90}, draw=none, from=0, to=1]
\end{tikzcd}
\end{equation}
\end{example}
When $Q$ is coaugmented (there is a canonical coaugmentation when $Q$ is $1$-reduced), we also have an coaugmentation map $\Triv \to Q$, giving us another adjunction:
\begin{example}[Trivial-primitives adjunction]
    Fix $A \in \Set, \V \in \CAlg(\Prl), \C \in \CAlg(\PrlV)$, following \ref{notation:AVC}.
    Suppose we're given a cooperad $Q \in \Op_A(\V)$. Let $Q$ have an augmentation $\eta: Q \to \Triv$. We apply \eqref{eqn:coalgadj} to $\eta$ and get the {\it trivial coalgebra-primitives} adjunction:
    \[
    \begin{tikzcd}
	\Fun(A, \C) & \Coalg_{Q}.
	\arrow["\Triv"{name=0}, shift left=2, from=1-1, to=1-2]
	\arrow["\Prim"{name=1},shift left=2, from=1-2, to=1-1]
 \arrow["\dashv"{anchor=center, rotate=-90}, draw=none, from=0, to=1]
    \end{tikzcd}
    \]
\end{example}

\subsection{(Co)modules over (co)operads}

Next we define a notion of module and comodule over a (co)operad. To do this we show there is a symmetric monoidal functor $\Md: \Sseq_A \to \Sseq_{A \amalg A}$. Intuitively it takes an operad $P$ to a related operad $\Md_P$ whose algebras are pairs of an algebra and a module over it. 

For the sake of clarity, instead of working with $A \amalg A$, we will denote the second copy of $A$ by $dA$, which are supposed to represent some infinitesimal approximation to the "generators" in $A$. Hence $\Md_Q$ will be an object of $\Sseq_{A \amalg dA}$. So we fix once and for all a canonical identification $d: A \to dA$.

From now on we also only work with symmetric sequences over {\it stable} presentably symmetric monoidal categories (so $\V \in \CAlg(\Prlst))$, although perhaps having a zero object is enough for our definitions. 
\begin{notation}[Convention on $A, \V, \C$]\label{notation:AVC2}
   We fix a set $A$, $\V \in \CAlg(\Prlst)$, and $\C \in \CAlg(\PrlstV)$. (Co)operads will be $A$-colored and over $\V$, and (co)algebras will be taken over $\C$. Notice that we just added stability assumptions to \ref{notation:AVC}.
\end{notation}

Let's begin with some notation. Let $\C^{\ast}$ denote the free addition of a zero object into $\C$. Then we construct a map of $1$-categories 
\[
\freesym{A \amalg dA}\times (A\amalg dA) \to (\freesym{A} \times A)^{\ast}
\]
as follows:
\begin{construction}\label{constr:premod}
Let $A$ be a set. Let $D: \freesym{A \amalg dA}\times (A\amalg dA) \to (\freesym{A} \times A)^{\ast}$ be constructed as follows: 
\begin{enumerate}
    \item Objects of the form $(\{a_1, \dots, a_n\}, a)$ with all $a_i \in A$ and $a \in A$ are sent to $(\{a_1, \dots, a_n\}, a)$.
    \item Objects of the form $(\{a_1,\dots da_i \dots, a_n\}, da)$ where the domain has exactly one object $da_i \in dA$, and codomain $da \in dA$ are sent to $(\{a_1,\dots a_i \dots, a_n\}, a)$ (here one uses the inverse of the canonical identification $d: A \to dA$). 
    \item All other objects are sent to $0$, the zero object.
\end{enumerate}
On morphisms, one sends the automorphisms to the obvious ones (either the same one in the first case, the natural induced one using $d$ in the second case, or the zero map in the third case). 
\end{construction}

Notice that as long as $\V$ has a zero object, we have \[
\Fun^{\ast}([\freesym{A} \times A)]^{\ast}, \V) \simeq \Sseq_A.
\] Hence we can define $\Md$ as follows:

\begin{definition}
    Fix a set $A$ and $\V \in \CAlg(\Prlst)$.
    Let $\Md: \Sseq_A \to \Sseq_{A \amalg dA}$ be defined by restriction along $D$(\ref{constr:premod}). Explicitly, 
    \[
    \begin{tikzcd}
        \Sseq_A \ar[r, "\simeq"] & \Fun^{\ast}([\freesym{A} \times A)]^{\ast}, \V) \ar[r, "D^{*}"] & \Sseq_{A \amalg dA},
    \end{tikzcd}
    \]
    where we're using the definition of $\Sseq$ (\ref{def:sseq})
\end{definition}

Next we check that this functor is symmetric monoidal by calculating it explicitly with formulas. Unfortunately we weren't able to find a natural argument using universal properties. 

\begin{lemma}\label{lemma:modmonoidal}
Fix a set $A$ and $\V \in \CAlg(\Prlst)$. The functor $\Md: \Sseq_A \to \Sseq_{A \amalg dA}$ is monoidal with respect to the circle product (or $\circ$-monoidal).
\end{lemma}

\begin{proof}
Given two symmetric sequences, we need to show that $\Md_{X \circ Y} \simeq \Md_X \circ \Md_Y$.
This just amounts to a colimit calculation: we use our colimit formula to evaluate $\Md_X \circ \Md_Y(\{a_1, \dots, a_n\}, a)$ as
\[
\bigoplus_{\{b_1, \dots, b_m\}} \Md_X(\{b_1, \dots, b_m\}, a) \otimes_{\Sigma_m} [\Md_Y(\{a_1, \dots, a_{i_1}\}, b_1) \otimes \dots \otimes \Md_Y(\{a_{i_{m+1}}, \dots, a_n\}, b_m)].
\]

Now notice that $\Md_X(\{b_1, \dots, b_m\}, a)$ is zero unless all the $b_i \in A$, in which case this evaluates to $X(\{b_1, \dots, b_m\}, a)$, similar with all the $\Md_Y$ terms. Thus this evaluates to 
\[
\bigoplus_{\{b_1, \dots, b_m\}} X(\{b_1, \dots, b_m\}, a) \otimes_{\Sigma_m} [Y(\{a_1, \dots, a_{i_1}\}, b_1) \otimes \dots \otimes Y(\{a_{i_{m+1}}, \dots, a_n\}, b_m)] \simeq X\circ Y (\{a_1, \dots, a_n\}, a).
\]

But this agrees by definition with $\Md_{X \circ Y}(\{a_1, \dots, a_n\}, a)$.

One can do a similar argument for when you evaluate $\Md_X \circ \Md_Y (\{a_1, \dots da_i \dots a_n\}, da)$, and it again agrees with $\Md_{X\circ Y} (\{a_1, \dots da_i \dots a_n\}, da)$. On all other inputs, both sides evalute to $0$.

Now since our equivalences are identifications of colimits, they are natural and give us the desired equivalence $\Md_{X \circ Y} \simeq \Md_X \circ \Md_Y$.
\end{proof}
\begin{remark}
    There are various other ways to see this result, but most of them seem to either use a colimit identification or an identification of left Kan extension, both of which in essence is this calculation above. 
    %Search online if there is a better way to see this!
\end{remark}
\begin{notation}
Fix a set $A$, $\V \in \CAlg(\Prlst)$, and $\C \in \CAlg(\PrlstV)$. Then let $S \in \Sseq_A(\V)$. Given a pair $C \in \V$ and $M \in \V$, we use $S \circ (C;M)$ to denote the second factor of $\Md_S \circ (C, M)$. In other words, 
\[
\Md_S \circ (C, M) = (S\circ V, S \circ (C;M)).
\]
This notation follows the conventions of Loday and Vallette \cite[Section~6.1.1]{LV}.
\end{notation}

Now we can talk about (co)modules for (co)operads:
\begin{definition}[Modules and comodules]
    Fix a set $A$, $\V \in \CAlg(\Prlst)$, and $\C \in \CAlg(\PrlstV)$ following \ref{notation:AVC2}.

    Given $P$ an $A$-operad over $\V$. Let $\Md_P:= \Md(P)$ be the {\it operad of modules over $P$}. This is an operad by (\ref{lemma:modmonoidal}).
    Further, let $\Mod_P(\C)$ or $\Mod_P$ denote the category of algebras over $\Md_P$. 

    Similarly let $Q$ be an $A$-cooperad over $\V$. Then we let $\Comod_Q := \Md(Q)$ be the {\it cooperad of comodules over $Q$}. Again, this is a cooperad by (\ref{lemma:modmonoidal}). And let $\Comod_Q(\C)$ or $\Comod_Q$ denote the category of coalgebras over $\Md_Q$.
    
\end{definition}
\subsection{External direct sum}
To aid in our analysis of module and comodule categories, we now define an external direct sum. Given two sets $A, B$, we now construct a functor of external direct sum from $\Sseq_A \times \Sseq_B \to \Sseq_{A\amalg B}$. 

\begin{construction}[External direct sum]\label{constr:sum}
    Given two sets $A, B$ and $\V \in \CAlg(\Prl)$, we have natural restriction maps coming from the inclusions of $A \to A \amalg B$, inducing $\freesym{A} \times A \to \freesym{A \amalg B} \times (A \amalg B)$, and similar with $B$. Thus we have the restriction functor 
    \[
    \res: \Sseq_{A \amalg B} \to \Sseq_A \times \Sseq_B
    \]
    We can also left/right Kan extend along the inclusion 
    \[
    (\freesym{A} \times A) \amalg (\freesym{B} \times B) \to \freesym{A \amalg B} \times (A \amalg B)
    \]
    to get the same functor ${-} \oplus {-}$ of extension by zero, or "external direct sum"
    \[
    \oplus: \Sseq_A \times \Sseq_B \to \Sseq_{A \amalg B}.
    \] This is both a left and right adjoint to the rectriction functor.
    
    The left and right Kan extensions agree because $\V$ has an initial object, and because $\freesym{A \amalg B}$ is a very simple category with no morphisms between two different objects. 
\end{construction}

Now one can see the following:
\begin{lemma}
    The functor $\oplus$ as constructed (\ref{constr:sum}) is strongly $\circ$-monoidal.
\end{lemma}
\begin{proof}
Suppose we're given $S, S' \in \Sseq_A$ and $T, T' \in \Sseq_B$. Then for any $a_1,\dots, a_n, a \in A$, we have
\[
(S \circ S') \oplus (T \circ T')(\{a_1, \dots, a_n; a\}) \simeq S \circ S' (\{a_1, \dots, a_n; a\})
\]
by calculating the extension by zero. Similarly, one gets the same answer by calculating 
\[
(S \oplus T) \circ (S' \oplus T') (\{a_1, \dots, a_n; a\})\simeq S \circ S'(\{a_1, \dots, a_n; a\}).
\]
The same situation happens with $b_1, \dots, b_n, b \in B$. For any mixed term with $a_i \in A$ and $b_j \in B$ simultaneously, both $(S \circ S') \oplus (T \circ T')$ and $(S \oplus T) \circ (S' \oplus T')$ just evaluate to zero. Since these equivalences are equivalences of colimits, they are natural and give an equivalence 
\[
(S \circ S') \oplus (T \circ T') \simeq (S \oplus T) \circ (S' \oplus T').
\]
\end{proof}

This also proves:
\begin{corollary}
    The restriction functor $\res$ as constructed in (\ref{constr:sum}) is both lax and colax $\circ$-monoidal.
\end{corollary}

Hence, we have the following adjunction for operads:
    \begin{equation}\label{eqn:opsum}
    \begin{tikzcd}
	\Op_A \times \Op_B & \Op_{A \amalg B}
	\arrow["\oplus"{name=0}, shift left=2, from=1-1, to=1-2]
	\arrow["\res"{name=1},shift left=2, from=1-2, to=1-1]
    \arrow["\dashv"{anchor=center, rotate=-90}, draw=none, from=0, to=1]
    \end{tikzcd}
    \end{equation}
and the following for cooperads:
\begin{equation}\label{eqn:coopsum}
    \begin{tikzcd}
	\CoOp_{A \amalg B} & \CoOp_A \times \CoOp_B
	\arrow["\res"{name=0}, shift left=2, from=1-1, to=1-2]
	\arrow["\oplus"{name=1},shift left=2, from=1-2, to=1-1]
    \arrow["\dashv"{anchor=center, rotate=-90}, draw=none, from=0, to=1]
    \end{tikzcd}
    \end{equation}
\begin{remark}\label{rmk:sum}
    Note that coalgebras over $Q \oplus Q'$ for cooperads $Q \in \CoOp_A$, $Q' \in \CoOp_B$ consist of a pair of a $Q$-coalgebra and a $Q'$-coalgebra. Indeed we see that 
    \[
    \Coalg_{Q \oplus Q'} \simeq \Coalg_Q \times \Coalg_{Q'},
    \] and an analoguous statement is true for operads and algebras.
\end{remark}

\subsection{(Co)modules as a fibration}
Recall the situation in commutative rings, where we have a bifibration $\Md \to \Rings$ such that over each ring $R$, the fiber is $\Md_R$. We now prove this result for our $\Md_P$ and $\Comod_Q$ defined over  cooperads.

To do so, we'll need the following lemma:

\begin{lemma}\label{lemma:bifibration}
Given functors $\pi: \C \to \D$ and $l: \D \to \C$ such that $l$ is fully faithful left adjoint to $\pi$. Let $\C$ have pushouts, and let $\pi$ preserve pushouts. Then $p$ is a coCartesian fibration.
\end{lemma}

\begin{proof}
    We notice that to check that $\pi$ is a coCartesian fibration, we must check that given $c \in \C$ and $f: \pi c = d \to d'$ in $\D$, there exists $\hat{f}: c \to c'$ in $\C$ lifting $f$ such that given any $c'' \in C$ with $d'' = \pi c''$, the following square 
% https://q.uiver.app/#q=WzAsNCxbMCwwLCJcXEMoYycsIGMnJykiXSxbMSwwLCJcXEMoYywgYycnKSJdLFsxLDEsIlxcRChkLCBkJycpIl0sWzAsMSwiXFxEKGQnLCBkJycpIl0sWzAsMSwiXFxoYXR7Zn1eXFxhc3QiXSxbMSwyLCJcXHBpIiwyXSxbMCwzLCJcXHBpIiwyXSxbMywyXSxbMCwyLCIiLDEseyJzdHlsZSI6eyJuYW1lIjoiY29ybmVyIn19XV0=
\[\begin{tikzcd}
	{\C(c', c'')} & {\C(c, c'')} \\
	{\D(d', d'')} & {\D(d, d'')}
	\arrow["{\hat{f}^\ast}", from=1-1, to=1-2]
	\arrow["\pi", from=1-2, to=2-2]
	\arrow["\pi", from=1-1, to=2-1]
	\arrow["{f^\ast}", from=2-1, to=2-2]
	\arrow["\lrcorner"{anchor=center, pos=0.125}, draw=none, from=1-1, to=2-2]
\end{tikzcd}\]
is a pullback. The $2$-cell showing commutativity here is from the functoriality of $\pi$. 

To do so, we first construct $\hat{f}$ as the following pullback
% https://q.uiver.app/#q=WzAsNCxbMCwwLCJsZCJdLFsxLDAsImxkJyJdLFswLDEsImMiXSxbMSwxLCJjJyJdLFswLDEsImxmIl0sWzAsMiwiXFxlcHNpbG9uX2MiLDJdLFsxLDNdLFsyLDMsIlxcaGF0e2Z9IiwyXSxbMywwLCIiLDEseyJzdHlsZSI6eyJuYW1lIjoiY29ybmVyIn19XV0=
\begin{equation}\label{eqn:hatf_defn}
\begin{tikzcd}
	ld & {ld'} \\
	c & {c'}
	\arrow["lf", from=1-1, to=1-2]
	\arrow["{\epsilon_c}", from=1-1, to=2-1]
	\arrow[from=1-2, to=2-2]
	\arrow["{\hat{f}}"', from=2-1, to=2-2]
	\arrow["\lrcorner"{anchor=center, pos=0.125, rotate=180}, draw=none, from=2-2, to=1-1]
\end{tikzcd}
\end{equation}
where $\epsilon_c$ is the counit on $c$ (we're using $ld = l\pi c$).

Then we first check that $\pi \hat{f} \simeq f$ by taking $\pi$ of the diagram \ref{eqn:hatf_defn}. Then we get the following pushout (as $\pi$ preserves pushouts)

% https://q.uiver.app/#q=WzAsNixbMCwxLCJcXHBpIGxkIl0sWzEsMSwiXFxwaSBsZCciXSxbMCwyLCJcXHBpIGMiXSxbMSwyLCJcXHBpIGMnIl0sWzAsMCwiZCJdLFsxLDAsImQnIl0sWzAsMSwiXFxwaSBsZiJdLFswLDIsIlxccGlcXGVwc2lsb25fYyJdLFsxLDNdLFsyLDMsIlxccGlcXGhhdHtmfSIsMl0sWzMsMCwiIiwxLHsic3R5bGUiOnsibmFtZSI6ImNvcm5lciJ9fV0sWzQsMCwiXFxldGFfZCJdLFs1LDEsIlxcZXRhX3tkJ30iXSxbNCw1LCJmIl1d
\[\begin{tikzcd}
	d & {d'} \\
	{\pi ld} & {\pi ld'} \\
	{\pi c} & {\pi c'}
	\arrow["{\pi lf}", from=2-1, to=2-2]
	\arrow["{\pi\epsilon_c}", from=2-1, to=3-1]
	\arrow[from=2-2, to=3-2]
	\arrow["{\pi\hat{f}}"', from=3-1, to=3-2]
	\arrow["\lrcorner"{anchor=center, pos=0.125, rotate=180}, draw=none, from=3-2, to=2-1]
	\arrow["{\eta_d}", from=1-1, to=2-1]
	\arrow["{\eta_{d'}}", from=1-2, to=2-2]
	\arrow["f", from=1-1, to=1-2]
\end{tikzcd}\]
where $\eta$ is the unit. Note that the left composite $\epsilon_c \eta_d \simeq 1_c$ by the triangle identity (recall $d = \pi c$). These units are also equivalences since $l$ is fully faithful. Hence the outer square is also a pullback, and thus we have
\begin{equation}\label{eqn:2cell}
\begin{tikzcd}
	d & {d'} \\
	{\pi ld} & {\pi ld'} \\
	{\pi c} & {\pi c'}
	\arrow["{\pi lf}", from=2-1, to=2-2]
	\arrow["{\pi\epsilon_c}", from=2-1, to=3-1]
	\arrow[from=2-2, to=3-2]
	\arrow["{\pi\hat{f}}"', from=3-1, to=3-2]
	\arrow["\lrcorner"{anchor=center, pos=0.125, rotate=180}, draw=none, from=3-2, to=2-1]
	\arrow["{\eta_d}", from=1-1, to=2-1]
	\arrow["{\eta_{d'}}", from=1-2, to=2-2]
	\arrow["f", from=1-1, to=1-2]
\end{tikzcd}
\simeq
\begin{tikzcd}
	d & {d'} \\
	\pi c & {\pi c'}
	\arrow["f", from=1-1, to=1-2]
	\arrow[equal, from=1-1, to=2-1]
	\arrow[equal, from=1-2, to=2-2]
	\arrow["{f}"', from=2-1, to=2-2]
\end{tikzcd}
\end{equation} where we're using that the pushout of an identity is another identity. The $2$-cell filling the square here is the identity of $f$. Hence we've shown that $\pi \hat{f} \simeq f$ as required.

Next we show that the pullback condition is fulfilled. Suppose we're given $c'' \in \C$ with $d'' = \pi c''$. Then by definition of $c'$ as a pullback, we have the following diagram
% https://q.uiver.app/#q=WzAsNixbMCwwLCJcXEMoYycsIGMnJykiXSxbMSwwLCJcXEMoYywgYycnKSJdLFsxLDEsIlxcQyhsZCwgYycnKSJdLFswLDEsIlxcQyhsZCcsIGMnJykiXSxbMCwyLCJcXEQoZCcsIGQnJykiXSxbMSwyLCJcXEQoZCwgZCcnKSJdLFswLDEsIlxcaGF0e2Z9XlxcYXN0Il0sWzEsMl0sWzAsM10sWzMsMl0sWzAsMiwiIiwxLHsic3R5bGUiOnsibmFtZSI6ImNvcm5lciJ9fV0sWzMsNCwiXFxzaW1lcSJdLFsyLDUsIlxcc2ltZXEiXSxbNCw1LCJmXlxcYXN0IiwyXV0=
\begin{equation}\label{eqn:composite2cell}
\begin{tikzcd}
	{\C(c', c'')} & {\C(c, c'')} \\
	{\C(ld', c'')} & {\C(ld, c'')} \\
	{\D(d', d'')} & {\D(d, d'')}.
	\arrow["{\hat{f}^\ast}", from=1-1, to=1-2]
	\arrow[from=1-2, to=2-2]
	\arrow[from=1-1, to=2-1]
	\arrow[from=2-1, to=2-2]
	\arrow["\lrcorner"{anchor=center, pos=0.125}, draw=none, from=1-1, to=2-2]
	\arrow["\simeq", from=2-1, to=3-1]
	\arrow["\simeq", from=2-2, to=3-2]
	\arrow["{f^\ast}"', from=3-1, to=3-2]
\end{tikzcd}
\end{equation}
Notice that the equivalences on the bottom square come from the adjunction. Since the bottom square has equivalences as legs, the outer square is also a pullback. 

Now the outer composite square composes exactly to 
\[\begin{tikzcd}
	{\C(c', c'')} & {\C(c, c'')} \\
	{\D(d', d'')} & {\D(d, d'')}
	\arrow["{\hat{f}^\ast}", from=1-1, to=1-2]
	\arrow["\pi", from=1-2, to=2-2]
	\arrow["\pi", from=1-1, to=2-1]
	\arrow["{f^\ast}", from=2-1, to=2-2]
	\arrow["\lrcorner"{anchor=center, pos=0.125}, draw=none, from=1-1, to=2-2]
\end{tikzcd}\]
which is compatible with the $2$-cells. This follows from the equation \ref{eqn:2cell} and the functoriality of $\pi$ on $2$-cells. We argue briefly as follows. The diagram \ref{eqn:composite2cell} factors as follows:
% https://q.uiver.app/#q=WzAsOCxbMCwxLCJcXEMobGQnLGMnJykiXSxbMSwxLCJcXEMobGQnJywgYycnKSJdLFswLDIsIlxcRChcXHBpIGxkJywgZCcnKSJdLFsxLDIsIlxcRChcXHBpIGxkLCBkJycpIl0sWzAsMywiXFxEKGQnLCBkJycpIl0sWzEsMywiXFxEKGQsIGQnJykiXSxbMCwwLCJcXEMoYycsYycnKSJdLFsxLDAsIlxcQyhjLGMnJykiXSxbMCwxLCJsZl5cXGFzdCJdLFswLDIsIlxccGkiXSxbMSwzLCJcXHBpIl0sWzIsMywiXFxwaSBsIGZeXFxhc3QiXSxbMiw0LCJcXGV0YV97ZCd9XlxcYXN0Il0sWzMsNSwiXFxldGFfe2R9XlxcYXN0Il0sWzQsNSwiZl5cXGFzdCIsMl0sWzYsMF0sWzYsN10sWzcsMV0sWzcsMCwiXFxDKGgsYycnKSIsMSx7ImxldmVsIjoyfV1d
\[\begin{tikzcd}
	{\C(c',c'')} & {\C(c,c'')} \\
	{\C(ld',c'')} & {\C(ld'', c'')} \\
	{\D(\pi ld', d'')} & {\D(\pi ld, d'')} \\
	{\D(d', d'')} & {\D(d, d'')}
	\arrow["{lf^\ast}", from=2-1, to=2-2]
	\arrow["\pi", from=2-1, to=3-1]
	\arrow["\pi", from=2-2, to=3-2]
	\arrow["{\pi l f^\ast}", from=3-1, to=3-2]
	\arrow["{\eta_{d'}^\ast}", from=3-1, to=4-1]
	\arrow["{\eta_{d}^\ast}", from=3-2, to=4-2]
	\arrow["{f^\ast}"', from=4-1, to=4-2]
	\arrow[from=1-1, to=2-1]
	\arrow[from=1-1, to=1-2]
	\arrow[from=1-2, to=2-2]
	\arrow["{\C(h,c'')}"{description}, Rightarrow, from=1-2, to=2-1]
\end{tikzcd}\]
where the upper square comes from $\C(-, c'')$ applied to the $2$-cell $h$ in the definition of $c'$ as a pullback. One can commute the middel $2$-cell with $\C(h, c'')$, using the functoriality of $\pi$ to get 
% https://q.uiver.app/#q=WzAsOCxbMCwxLCJcXEQoZCcsIGQnJykiXSxbMSwxLCJcXEQoZCwgZCcnKSJdLFswLDIsIlxcRChcXHBpIGxkJywgZCcnKSJdLFsxLDIsIlxcRChcXHBpIGxkLCBkJycpIl0sWzAsMywiXFxEKGQnLCBkJycpIl0sWzEsMywiXFxEKGQsIGQnJykiXSxbMCwwLCJcXEMoYycsYycnKSJdLFsxLDAsIlxcQyhjLGMnJykiXSxbMCwxLCJsZl5cXGFzdCJdLFswLDJdLFsxLDNdLFsyLDMsIlxccGkgbCBmXlxcYXN0IiwyXSxbMiw0LCJcXGV0YV97ZCd9XlxcYXN0Il0sWzMsNSwiXFxldGFfe2R9XlxcYXN0Il0sWzQsNSwiZl5cXGFzdCIsMl0sWzYsMCwiXFxwaSJdLFs2LDddLFs3LDEsIlxccGkiXSxbMSwyLCJcXEQoXFxwaSBoLCBkJycpIiwxLHsibGV2ZWwiOjJ9XV0=
\[\begin{tikzcd}
	{\C(c',c'')} & {\C(c,c'')} \\
	{\D(d', d'')} & {\D(d, d'')} \\
	{\D(\pi ld', d'')} & {\D(\pi ld, d'')} \\
	{\D(d', d'')} & {\D(d, d'')}.
	\arrow["{lf^\ast}", from=2-1, to=2-2]
	\arrow[from=2-1, to=3-1]
	\arrow[from=2-2, to=3-2]
	\arrow["{\pi l f^\ast}"', from=3-1, to=3-2]
	\arrow["{\eta_{d'}^\ast}", from=3-1, to=4-1]
	\arrow["{\eta_{d}^\ast}", from=3-2, to=4-2]
	\arrow["{f^\ast}"', from=4-1, to=4-2]
	\arrow["\pi", from=1-1, to=2-1]
	\arrow[from=1-1, to=1-2]
	\arrow["\pi", from=1-2, to=2-2]
	\arrow["{\D(\pi h, d'')}"{description}, Rightarrow, from=2-2, to=3-1]
\end{tikzcd}\]
Now the middle and bottom square then cancels because of \ref{eqn:2cell}. Thus we left with the upper $2$-cell which is given by functoriality of $\pi$, as required. 
\end{proof}

Now we can define the map $p: \Comod_Q \to \Coalg_Q$. Then we'll use the above lemma to show it is a coCartesian and cartesian fibration.

\begin{construction}[Fibration of comodules]\label{constr:modfib}
Fix a set $A$, $\V \in \CAlg(\Prlst)$, and $\C \in \CAlg(\PrlstV)$ following \ref{notation:AVC2}. Suppose we're given a cooperad $Q$. We use the map of cooperads $\Md_Q \to Q \oplus \Triv$ which is given by the universal property of $Q \oplus \Triv$: note that $\res(\Md_Q)$ is exactly $(Q, \Triv)$ (see \ref{constr:sum}). 

Hence we have a natural map $\Md_Q \to Q \oplus \Triv$, inducing a map of coalgebras (the left adjoint of \ref{eqn:coalgadj})
\begin{equation}\label{eqn:mod2coalg1}
\Comod_Q \to \Coalg_Q \times \Fun(dA, \C).
\end{equation}
We can project to $\Coalg_Q$, thus getting the required map
\[
p: \Comod_Q \to \Coalg_Q.
\]

Another perspective on this construction can be given as follows: let $\Sseq_A$ act on $\Fun(A \amalg dA, \C)$ by $S \mapsto (\Md_S \circ {-})$. Then $\Fun(A \amalg dA, \C)$ and $\Fun(A, \C)$ are both categories with $\Sseq_A$ action. Notice that $\pi_1: \Fun(A \amalg dA, \C) \to \Fun(A, \C)$, the functor only remembering the first $A$ variables, is strongly $\Sseq_A$-monoidal. In other words, we have 
\[
\alpha: \pi_1 (\Md_S \circ {-}) \simeq S \circ {\pi_1{-}}
\]
which is natural in $S$ and respects the composition product. Hence given a cooperad $Q$, we get an induced map on coalgebras 
\[
p: \Comod_Q \to \Coalg_Q
\]
which lifts this map $\pi_1$ over the forgetful functors. Notice also that $\pi_1$ has both a left and right adjoint $(-,0): \Fun(A, \C) \to \Fun(A \amalg dA, \C)$. Since $\pi_1$ is strongly $\Sseq_A$-monoidal (so the lax/colax morphism are equivalences), its right adjoint $(-,0)$ is lax and colax $\Sseq_A$-monoidal \ref{def:lax}. Hence, it lifts to a right adjoint 
\[
r: \Coalg_Q \to \Comod_Q
\] over the forgetful functors.
\end{construction}

\begin{remark}\label{remark:pcofree}
    Notice that by how $p, r$ are defined in \ref{constr:modfib}, we have clearly that $p \Cofree_{\Md_Q} \simeq \Cofree_Q \pi_1$ lifting the natural strong $\Sseq_A$-monoidality $\pi_1$, ie the map $\pi_1 \Md_Q \circ{-} \to Q \circ {\pi_1{-}}$.  
\end{remark}
\begin{comment}
\begin{proof}
    This is a simple diagram chase: we just show that $U \beta$ is an equivalence. To do so, we use the fact that $\beta$ is an adjunct, so it can be calculated as
    \[
    \begin{tikzcd}
    {p \Cofree_{\Md_Q}} \ar[d, "\eta_Q"] \\
    {\Cofree_{Q} U_Q p \Cofree_{\Md_Q}} \ar[r, "\simeq"] & {\Cofree_{Q} \pi_1 \Md_Q \circ{-}} \ar[r, "\Cofree_{Q} \alpha"] & {\Cofree_{Q} Q \circ {\pi_1{-}}} \ar[r, "\Cofree_{\Md_Q} \varepsilon_Q"] & \Cofree_{\Md_Q} \pi_1
    \end{tikzcd}
    \]

    Applying $U$, we get
    \[
    \begin{tikzcd}
    {p \Cofree_{\Md_Q}} \ar[d, "\eta_Q"] \\
    {\Cofree_{Q} U_Q p \Cofree_{\Md_Q}} \ar[r, "\simeq"] & {\Cofree_{Q} \pi_1 \Md_Q \circ{-}} \ar[r, "\Cofree_{Q} \alpha"] & {\Cofree_{Q} Q \circ {\pi_1{-}}} \ar[r, "\Cofree_{\Md_Q} \varepsilon_Q"] & \Cofree_{\Md_Q} \pi_1
    \end{tikzcd}
    \]
\end{proof}
\end{comment}
By taking fibers of $p$, we can now discuss comodules over a coalgebra.
\begin{definition}
Fix a set $A$, $\V \in \CAlg(\Prlst)$, and $\C \in \CAlg(\PrlstV)$ following \ref{notation:AVC2}. Suppose we're given a cooperad $Q$ along with a coalgebra $C$. Then we define $\Comod_C$ to be the fiber of $p: \Comod_Q \to \Coalg_Q$ over $C$. This is called the the {\it category of comodules over $C$}.
\end{definition}

Next we show that $p: \Comod_Q \to \Coalg_Q$ is indeed a bifibration.
\begin{lemma}
    Fix a set $A$, $\V \in \CAlg(\Prlst)$, and $\C \in \CAlg(\PrlstV)$ following \ref{notation:AVC2}. Suppose we're given a cooperad $Q$. Then the map $p: \Comod_Q \to \Coalg_Q$ as constructed in (\ref{constr:modfib}) is both a cartesian and coCartesian fibration.
\end{lemma}

\begin{proof}
    We show this by verifying the conditions of (\ref{lemma:bifibration}). Thus we now show that $p$ has a left and right adjoint, both of which are intuitively equal to the map $C \mapsto (C, 0)$. The fact that the right adjoint exists is a formality because both the projection and the map \eqref{eqn:mod2coalg1} have right adjoints. We use $r$ to denote the right adjoint of $p$. Notice that $r$ can be calculated on cofree algebras using the following diagram:

% https://q.uiver.app/#q=WzAsNixbMCwwLCJcXENvbW9kX1EiXSxbMSwwLCJcXENvYWxnX1EgXFx0aW1lcyBcXFYiXSxbMiwwLCJcXENvYWxnX1EiXSxbMiwxLCJcXFYiXSxbMSwxLCJcXFYgXFx0aW1lcyBcXFYiXSxbMCwxLCJcXFYgXFx0aW1lcyBcXFYiXSxbMCwxLCJVIiwwLHsib2Zmc2V0IjotMn1dLFsxLDIsIlxccGlfMSIsMCx7Im9mZnNldCI6LTJ9XSxbMiwzLCJVX1EiLDAseyJvZmZzZXQiOi0yfV0sWzMsMiwiIiwwLHsib2Zmc2V0IjotMn1dLFsyLDEsIigtLCAwKSIsMCx7Im9mZnNldCI6LTJ9XSxbMSwwLCIiLDAseyJvZmZzZXQiOi0yfV0sWzEsNCwiVV97USBcXG9wbHVzIFxcVHJpdn0iLDAseyJvZmZzZXQiOi0yfV0sWzQsMSwiIiwwLHsib2Zmc2V0IjotMn1dLFs0LDMsIlxccGlfMSIsMCx7Im9mZnNldCI6LTJ9XSxbMyw0LCIoLSwgMCkiLDAseyJvZmZzZXQiOi0yfV0sWzQsNSwiPSIsMix7InN0eWxlIjp7InRhaWwiOnsibmFtZSI6ImFycm93aGVhZCJ9fX1dLFs1LDAsIiIsMCx7Im9mZnNldCI6LTJ9XSxbMCw1LCJVX3tcXE1kX1F9IiwwLHsib2Zmc2V0IjotMn1dLFsxOCwxNywiIiwyLHsibGV2ZWwiOjEsInN0eWxlIjp7Im5hbWUiOiJhZGp1bmN0aW9uIn19XSxbMTIsMTMsIiIsMix7ImxldmVsIjoxLCJzdHlsZSI6eyJuYW1lIjoiYWRqdW5jdGlvbiJ9fV0sWzE0LDE1LCIiLDEseyJsZXZlbCI6MSwic3R5bGUiOnsibmFtZSI6ImFkanVuY3Rpb24ifX1dLFs3LDEwLCIiLDAseyJsZXZlbCI6MSwic3R5bGUiOnsibmFtZSI6ImFkanVuY3Rpb24ifX1dLFs2LDExLCIiLDAseyJsZXZlbCI6MSwic3R5bGUiOnsibmFtZSI6ImFkanVuY3Rpb24ifX1dLFs4LDksIiIsMix7ImxldmVsIjoxLCJzdHlsZSI6eyJuYW1lIjoiYWRqdW5jdGlvbiJ9fV1d
\[\begin{tikzcd}
	{\Comod_Q} & {\Coalg_Q \times \Fun(dA, \C)} & {\Coalg_Q} \\
	{\Fun(A \amalg dA, \C)} & {\Fun(A \amalg dA, \C)} & \Fun(A, \C).
	\arrow[""{name=0, anchor=center, inner sep=0}, "U", shift left=2, from=1-1, to=1-2]
	\arrow[""{name=1, anchor=center, inner sep=0}, "{\widehat{\pi_1}}", shift left=2, from=1-2, to=1-3]
	\arrow[""{name=2, anchor=center, inner sep=0}, "{U_Q}", shift left=2, from=1-3, to=2-3]
	\arrow[""{name=3, anchor=center, inner sep=0}, shift left=2, from=2-3, to=1-3]
	\arrow[""{name=4, anchor=center, inner sep=0}, "{\widehat{(-, 0)}}", shift left=2, from=1-3, to=1-2]
	\arrow[""{name=5, anchor=center, inner sep=0}, shift left=2, from=1-2, to=1-1]
	\arrow[""{name=6, anchor=center, inner sep=0}, "{U_{Q \oplus \Triv}}", shift left=2, from=1-2, to=2-2]
	\arrow[""{name=7, anchor=center, inner sep=0}, shift left=2, from=2-2, to=1-2]
	\arrow[""{name=8, anchor=center, inner sep=0}, "{\pi_1}", shift left=2, from=2-2, to=2-3]
	\arrow[""{name=9, anchor=center, inner sep=0}, "{(-, 0)}", shift left=2, from=2-3, to=2-2]
	\arrow["{=}"', <->, from=2-2, to=2-1]
	\arrow[""{name=10, anchor=center, inner sep=0}, shift left=2, from=2-1, to=1-1]
	\arrow[""{name=11, anchor=center, inner sep=0}, "{U_{\Md_Q}}", shift left=2, from=1-1, to=2-1]
	\arrow["\dashv"{anchor=center, rotate=-180}, draw=none, from=11, to=10]
	\arrow["\dashv"{anchor=center, rotate=-180}, draw=none, from=6, to=7]
	\arrow["\dashv"{anchor=center, rotate=-90}, draw=none, from=8, to=9]
	\arrow["\dashv"{anchor=center, rotate=-90}, draw=none, from=1, to=4]
	\arrow["\dashv"{anchor=center, rotate=-90}, draw=none, from=0, to=5]
	\arrow["\dashv"{anchor=center, rotate=-180}, draw=none, from=2, to=3]
\end{tikzcd}\]
Here the vertical adjunctions are forgetful/cofree adjunctions, and the horizontal ones calculate the map $p: \Comod_Q \to \Coalg_Q$. We note that $\Fun(A \amalg dA, \C) \simeq \Fun(A, \C) \times \Fun(dA, \C)$, hence our notation for $\pi_1$.

Now notice that on a cofree algebra $\Cofree(V)$, $r$ must take
\[
\Cofree_Q(V) \mapsto \Cofree_{\Md_Q}(V, 0) \simeq (\Cofree_Q(V), 0),
\]
and so by taking cobar resolutions via the cofree functor, we see that a coalgebra $C$ is taken to $(C, 0)$. So far everything was a formal consequence of categorical definitions. 
%Can explain the cobar resolution in more detail

Next we want to show that $r = ({-},0): \Coalg_Q \to \Comod_Q$ is also a left adjoint to $p: \Comod_Q \to \Coalg_Q$. For every $\mathbf{D} = (D, M)$, we need to show the following on homs:
\begin{equation}\label{eqn:hom1}
\Hom_{\Comod_Q}(rC, (D, M)) \simeq \Hom_{\Coalg_Q}(C, p(D,M)).
\end{equation}
We then attempt to take cofree resolutions of $(D, M)$. In order for this to work, we must show that $p$ preserves $U_{\Md_Q}$-split resolutions. This is pretty easy to show: let 
% https://q.uiver.app/#q=WzAsNixbMCwwLCJcXG1hdGhiZntEfSJdLFsxLDAsIlxcbWF0aGJme0R9XjAiXSxbMiwwLCJcXG1hdGhiZntEfV4xIl0sWzMsMCwiXFxkb3RzIl0sWzQsMCwiXFxtYXRoYmZ7RH1ebiJdLFs1LDAsIlxcZG90cyJdLFswLDFdLFsxLDIsIiIsMix7Im9mZnNldCI6MX1dLFsxLDIsIiIsMix7Im9mZnNldCI6LTF9XSxbMiwzLCIiLDIseyJvZmZzZXQiOi0yfV0sWzIsMywiIiwyLHsib2Zmc2V0IjoyfV0sWzIsM10sWzMsNCwiIiwyLHsib2Zmc2V0IjozfV0sWzMsNCwiIiwyLHsib2Zmc2V0IjotM31dLFs0LDUsIiIsMix7Im9mZnNldCI6LTN9XSxbNCw1LCIiLDIseyJvZmZzZXQiOjN9XSxbMTMsMTIsIiIsMix7InNob3J0ZW4iOnsic291cmNlIjoyMCwidGFyZ2V0IjoyMH19XSxbMTQsMTUsIiIsMix7InNob3J0ZW4iOnsic291cmNlIjoyMCwidGFyZ2V0IjoyMH19XV0=
\[\begin{tikzcd}
	{\mathbf{D}} & {\mathbf{D}^0} & {\mathbf{D}^1} & \dots & {\mathbf{D}^n} & \dots
	\arrow[from=1-1, to=1-2]
	\arrow[shift right, from=1-2, to=1-3]
	\arrow[shift left, from=1-2, to=1-3]
	\arrow[shift left=2, from=1-3, to=1-4]
	\arrow[shift right=2, from=1-3, to=1-4]
	\arrow[from=1-3, to=1-4]
	\arrow[""{name=0, anchor=center, inner sep=0}, shift right=3, from=1-4, to=1-5]
	\arrow[""{name=1, anchor=center, inner sep=0}, shift left=3, from=1-4, to=1-5]
	\arrow[""{name=2, anchor=center, inner sep=0}, shift left=3, from=1-5, to=1-6]
	\arrow[""{name=3, anchor=center, inner sep=0}, shift right=3, from=1-5, to=1-6]
	\arrow[shorten <=2pt, shorten >=2pt, phantom, "\dots"{rotate=90}, from=1, to=0]
	\arrow[shorten <=2pt, shorten >=2pt, phantom, "\dots"{rotate=90}, from=2, to=3]
\end{tikzcd}\]
be a $U_{\Md_Q}$-split totalization of $\mathbf{D} = (D, M)$. For brevity, we'll denote this totalization as 
\[
\mathbf{D} \simeq \Tot \mathbf{D}^\bullet.
\]
We want to check that 
\[\begin{tikzcd}
	{p \mathbf{D}} & {p \mathbf{D}^0} & {p \mathbf{D}^1} & \dots & {p \mathbf{D}^n} & \dots
	\arrow[from=1-1, to=1-2]
	\arrow[shift right, from=1-2, to=1-3]
	\arrow[shift left, from=1-2, to=1-3]
	\arrow[shift left=2, from=1-3, to=1-4]
	\arrow[shift right=2, from=1-3, to=1-4]
	\arrow[from=1-3, to=1-4]
	\arrow[""{name=0, anchor=center, inner sep=0}, shift right=3, from=1-4, to=1-5]
	\arrow[""{name=1, anchor=center, inner sep=0}, shift left=3, from=1-4, to=1-5]
	\arrow[""{name=2, anchor=center, inner sep=0}, shift left=3, from=1-5, to=1-6]
	\arrow[""{name=3, anchor=center, inner sep=0}, shift right=3, from=1-5, to=1-6]
	\arrow[shorten <=2pt, shorten >=2pt, phantom, "\dots"{rotate=90}, from=1, to=0]
	\arrow[shorten <=2pt, shorten >=2pt, phantom, "\dots"{rotate=90}, from=2, to=3]
\end{tikzcd}\]
is a $U_Q$-split totalization. To do so, let's apply $U_Q p \simeq  \pi_1 U_{\Md_Q}$ to $\Tot \mathbf{D}^\bullet$. Notice that  
\[\begin{tikzcd}
	{U_{\Md_Q} \mathbf{D}} & {U_{\Md_Q} \mathbf{D}^0} & {U_{\Md_Q} \mathbf{D}^1} & \dots & {U_{\Md_Q} \mathbf{D}^n} & \dots
	\arrow[from=1-1, to=1-2]
	\arrow[shift right, from=1-2, to=1-3]
	\arrow[shift left, from=1-2, to=1-3]
	\arrow[shift left=2, from=1-3, to=1-4]
	\arrow[shift right=2, from=1-3, to=1-4]
	\arrow[from=1-3, to=1-4]
	\arrow[""{name=0, anchor=center, inner sep=0}, shift right=3, from=1-4, to=1-5]
	\arrow[""{name=1, anchor=center, inner sep=0}, shift left=3, from=1-4, to=1-5]
	\arrow[""{name=2, anchor=center, inner sep=0}, shift left=3, from=1-5, to=1-6]
	\arrow[""{name=3, anchor=center, inner sep=0}, shift right=3, from=1-5, to=1-6]
	\arrow[shorten <=2pt, shorten >=2pt, phantom, "\dots"{rotate=90}, from=1, to=0]
	\arrow[shorten <=2pt, shorten >=2pt, phantom, "\dots"{rotate=90}, from=2, to=3]
\end{tikzcd}\]
is a split totalization by definition. Hence after applying $\pi_1$ of this, we also get a split totalization! Here we are also using that $\pi_1$ has both a left and right adjoint equal to $(-, 0)$, hence preserves all limits. Hence 
\[\begin{tikzcd}
	{U_Q p \mathbf{D}} & {U_Q p \mathbf{D}^0} & {U_Q p \mathbf{D}^1} & \dots & {U_Q p \mathbf{D}^n} & \dots
	\arrow[from=1-1, to=1-2]
	\arrow[shift right, from=1-2, to=1-3]
	\arrow[shift left, from=1-2, to=1-3]
	\arrow[shift left=2, from=1-3, to=1-4]
	\arrow[shift right=2, from=1-3, to=1-4]
	\arrow[from=1-3, to=1-4]
	\arrow[""{name=0, anchor=center, inner sep=0}, shift right=3, from=1-4, to=1-5]
	\arrow[""{name=1, anchor=center, inner sep=0}, shift left=3, from=1-4, to=1-5]
	\arrow[""{name=2, anchor=center, inner sep=0}, shift left=3, from=1-5, to=1-6]
	\arrow[""{name=3, anchor=center, inner sep=0}, shift right=3, from=1-5, to=1-6]
	\arrow[shorten <=2pt, shorten >=2pt, phantom, "\dots"{rotate=90}, from=1, to=0]
	\arrow[shorten <=2pt, shorten >=2pt, phantom, "\dots"{rotate=90}, from=2, to=3]
\end{tikzcd}\]
is equivalent to 
\[\begin{tikzcd}
	{\pi_1 U_{\Md_Q} \mathbf{D}} & {\pi_1 U_{\Md_Q} \mathbf{D}^0} & {\pi_1 U_{\Md_Q} \mathbf{D}^1} & \dots & {\pi_1 U_{\Md_Q} \mathbf{D}^n} & \dots
	\arrow[from=1-1, to=1-2]
	\arrow[shift right, from=1-2, to=1-3]
	\arrow[shift left, from=1-2, to=1-3]
	\arrow[shift left=2, from=1-3, to=1-4]
	\arrow[shift right=2, from=1-3, to=1-4]
	\arrow[from=1-3, to=1-4]
	\arrow[""{name=0, anchor=center, inner sep=0}, shift right=3, from=1-4, to=1-5]
	\arrow[""{name=1, anchor=center, inner sep=0}, shift left=3, from=1-4, to=1-5]
	\arrow[""{name=2, anchor=center, inner sep=0}, shift left=3, from=1-5, to=1-6]
	\arrow[""{name=3, anchor=center, inner sep=0}, shift right=3, from=1-5, to=1-6]
	\arrow[shorten <=2pt, shorten >=2pt, phantom, "\dots"{rotate=90}, from=1, to=0]
	\arrow[shorten <=2pt, shorten >=2pt, phantom, "\dots"{rotate=90}, from=2, to=3]
\end{tikzcd}\] which is a split totalization, as required! Hence $p$ preserves $U_{\Md_Q}$-split totalization, and in particular $\Md_Q$-cofree resolutions. 

So we can reduce to a cofree comodule $\mathbf{D} = (D, M) = \Cofree_{\Md_Q}(V, W)$. Using the above as well as \ref{remark:pcofree}, we can reduce to showing
\begin{equation}\label{eqn:hom2}
\Hom_{\Comod_Q}(rC, \Cofree_{\Md_Q}(V, W)) \simeq \Hom_{\Coalg_Q}(C,p \Cofree_{\Md_Q}(V, W) = \Cofree_Q(V)).
\end{equation}
By using forgetful/cofree adjunctions, we can further reduce to 
\begin{equation}\label{eqn:hom3}
\Hom_{\Fun(A \amalg dA, \C)}((UC, 0), (V, W)) \simeq \Hom_{\Fun(A, \C)}(UC, V)
\end{equation}
which follows because $(-, 0): \Fun(A, \C) \to \Fun(A \amalg dA, \C)$ is also left adjoint to $\pi_1: \Fun(A \amalg dA, \C) \to \Fun(A, \C)$. 

Lastly, we must check that $r$ is fully faithful. To do so, we check that the counit $\varepsilon: pr \to 1_{\Coalg_Q}$ is an equivalence. Notice that by the construction of $p$ and $r$ and the fact that they both commute with the forgetful functors (\ref{constr:modfib}), $U_Q \varepsilon$ is the counit of $\pi_1 \dashv (-, 0)$. This is an equivalence since $(-, 0)$ is a fully faithful biadjoint to $\pi_1$. Hence by conservativity of $U_Q$, we see that $\varepsilon: pr \to 1_{\Coalg_Q}$ is also an equivalence, so $r$ is fully faithful.

\begin{comment}
To do so, we check the unit 
\[\eta: 1_{\Coalg_Q} \to pr\] here is an equivalence. We can then reduce by cofree resolutions to the case of $C = \Cofree_Q(V)$. This uses the fact that since $p$ and $r$ are both right adjoints, they preserve totalizations, and hence cofree resolutions.

The unit $\eta$ applied to $\Cofree_Q(V)$ can be calculated using hom set calculations. Namely, let's plug in $(C,0) = (D,M) = \Cofree_{M_Q}(V,0)$ in the equation \eqref{eqn:hom1}:
\[
\Hom_{\Comod_Q}((\Cofree_Q(V), 0), (\Cofree_Q(V), 0)) \simeq \Hom_{\Coalg_Q}(\Cofree_Q(V), \Cofree_Q(V)).
\]
Now we'd like to chase where the identity of $(\Cofree_Q(V), 0)$ gets sent to on the right hand side. This is exactly the unit $\eta$.

Then following the logic above to \eqref{eqn:hom3} let's us factor the equivalence as:
\begin{align}
\Hom_{\Comod_Q}((\Cofree_Q(V), 0), (\Cofree_Q(V), 0)) & \simeq \Hom_{\Fun(A\amalg dA, \C)}((U\Cofree_Q(V), 0), (V, 0)) \\
& \simeq
\Hom_{\Fun(A, \C)}(U\Cofree_Q(V), V)\\ &\simeq
\Hom_{\Coalg_Q}(\Cofree_Q(V), \Cofree_Q(V)).
\end{align}
The first equivalence uses the $\Cofree_{\Md_Q} \dashv U_{\Md_Q}$ adjunction as well calculations of these functors applied to $(V, 0)$. We note that the identity of $(\Cofree_Q(V), 0)$ is mapped to the counit of the $U \dashv \Cofree_Q$ adjunction in the second row, which in turn is mapped to the identity of of $\Cofree_Q(V)$, as desired. Hence $\eta$ is an equivalence, as desired.
\end{comment}
Finally we can conclude using \ref{lemma:bifibration} as well as its dual: $p$ preserves pullbacks and pushouts as it has a left and right adjoint $r$, which is itself fully faithful. Also $\Comod_Q$ has pushouts and pullbacks. 
\end{proof}

We can now construct the pushforward and pullback of comodules.
\begin{definition}
Fix a set $A$, $\V \in \CAlg(\Prlst)$, and $\C \in \CAlg(\PrlstV)$ following \ref{notation:AVC2}.
Given a map $f: C \to D$ of coalgebras over a cooperad $Q$, our bifibration $p$ guarentees an adjunction
\[
\begin{tikzcd}
	\Comod_C & \Comod_D
	\arrow["f_{!}"{name=0}, shift left=2, from=1-1, to=1-2]
	\arrow["f^{\ast}"{name=1},shift left=2, from=1-2, to=1-1]
    \arrow["\dashv"{anchor=center, rotate=-90}, draw=none, from=0, to=1]
    \end{tikzcd}
\]
which we call pushforward and pullback of comodules. 
\end{definition}

\begin{example}
    Let $C \to k$ be the zero map from a $Q$-coalgbra to the zero coalgebra (thought of as co-augmented here). Then we have the following forgetful, cofree adjunction:
    \[
\begin{tikzcd}
	\Comod_C & \Fun(dA, \C).
	\arrow["U_C"{name=0}, shift left=2, from=1-1, to=1-2]
	\arrow["\Cofree_C"{name=1},shift left=2, from=1-2, to=1-1]
    \arrow["\dashv"{anchor=center, rotate=-90}, draw=none, from=0, to=1]
    \end{tikzcd}
\]
Given $V \in \Fun(A, \C)$ and $M \in \Fun(dA, \C)$, notice that 

\[
\Cofree_{\Cofree(V)}(M) \simeq \Cofree_{M_Q}(V, M)
\]
as objects in $\Comod_Q$, where $\Cofree_{\Cofree(V)}(M)$ is in the fiber $\Comod_C$. 
\end{example}

\begin{notation}
    We sometimes denote the induced comonad on $\Fun(dA, \C)$ by $M \mapsto Q \circ (C; M)$ in line with \cite[Section~6.1.1]{LV}. We also sometimes conflate the cofree object $\Cofree_C(M)$ with its underlying object $Q \circ (C;M).$
\end{notation}

\begin{remark}
    Everything in this section has an analog with modules over algebras, perhaps with small variation.
\end{remark}

\subsection{Stability of comodules}
Our strategy is as follows: we first construct the forgetful-cofree adjunction for comodules over $C$. Then we prove that it is comonadic and preserves limits over $\C$. Thus we have $\Comod_C$ gotten from an excisive comonad over $\C$, so it must also be stable when $\C$ is.

First we construct the forgetful-cofree adjunction for comodules over $C$. 
\begin{construction}[Forgetful-cofree adjunction of comodules]\label{constr:comod_cofree}
Fix a set $A$, $\V \in \CAlg(\Prlst)$, and $\C \in \CAlg(\PrlstV)$ following \ref{notation:AVC2}. Suppose we're given a cooperad $Q$ and a $Q$-coalgebra $C$. We construct an adjunction 
\[\begin{tikzcd}
	{\Comod_C} & \C
	\arrow[""{name=0, anchor=center, inner sep=0}, "{\Cofree_C}", shift left=2, from=1-2, to=1-1]
	\arrow[""{name=1, anchor=center, inner sep=0}, "{U_C}", shift left=2, from=1-1, to=1-2]
	\arrow["\dashv"{anchor=center, rotate=-90}, draw=none, from=1, to=0]
\end{tikzcd}\]
that fiberwise recovers the adjunction: 
% https://q.uiver.app/#q=WzAsMyxbMiwwLCJcXENvYWxnX1EgXFx0aW1lcyBcXEMiXSxbMCwwLCJcXENvbW9kX1EiXSxbMSwxLCJcXENvYWxnX1EiXSxbMSwyLCJwIiwyXSxbMCwyLCJcXHBpXzEiXSxbMSwwLCJVIiwwLHsib2Zmc2V0IjotMn1dLFswLDEsIlIiLDAseyJvZmZzZXQiOi0yfV0sWzUsNiwiIiwwLHsibGV2ZWwiOjEsInN0eWxlIjp7Im5hbWUiOiJhZGp1bmN0aW9uIn19XV0=
\[\begin{tikzcd}
	{\Comod_Q} && {\Coalg_Q \times \C}. \\
	& {\Coalg_Q}
	\arrow["p"', from=1-1, to=2-2]
	\arrow["{\pi_1}", from=1-3, to=2-2]
	\arrow[""{name=0, anchor=center, inner sep=0}, "U", shift left=2, from=1-1, to=1-3]
	\arrow[""{name=1, anchor=center, inner sep=0}, "R", shift left=2, from=1-3, to=1-1]
	\arrow["\dashv"{anchor=center, rotate=-90}, draw=none, from=0, to=1]
\end{tikzcd}\]
To do so we check that this is a relative adjunction \cite[Chapter~7.3.2]{HA}. We can do this by checking that 
\[
\pi_1\varepsilon: UR \to 1_{\Coalg_Q \times \C}
\] is an equivalence.

Notice that the counit is $UR \to 1$. Taking $\pi_1$, we get a map $pR \to \pi_1$. Notice that $\pi_1$ and $p$ preseve limits, hence it is enough to prove that $\pi_1 (UR \to 1)$ is an equivalence on $Q \oplus \Triv$-cofree objects by using cofree resolutions.

So we reduce to checking that 
\[
\pi_1 \varepsilon \Cofree_{Q \oplus \Triv}
\]
is an equivalence. We can check this after taking $U_Q$ since $U_Q$ is conservative. Hence we check 
\[
U_Q \pi_1 \varepsilon \Cofree_{Q \oplus \Triv} \simeq \pi_1 U_{Q \oplus \Triv} \varepsilon \Cofree_{Q \oplus \Triv}
\]
is an equivalence. This map however is equivalent to the first projection of the map induced by the cooperad morphism $\Md_Q \to Q \oplus \Triv$:
\[
\pi_1 (\Md_Q \circ {-} \to Q \oplus \Triv \circ {-})
\]
which is an equivalence on the first factor by construction! Hence this is a relative adjunction.

Thus being a relative adjunction, it gives us adjunctions on each fiber:
% https://q.uiver.app/#q=WzAsMixbMCwwLCJcXENvbW9kX0MiXSxbMSwwLCJcXEMiXSxbMSwwLCJcXENvZnJlZV9DIiwwLHsib2Zmc2V0IjotMn1dLFswLDEsIlVfQyIsMCx7Im9mZnNldCI6LTJ9XSxbMywyLCIiLDAseyJsZXZlbCI6MSwic3R5bGUiOnsibmFtZSI6ImFkanVuY3Rpb24ifX1dXQ==
\[\begin{tikzcd}
	{\Comod_C} & \C
	\arrow[""{name=0, anchor=center, inner sep=0}, "{\Cofree_C}", shift left=2, from=1-2, to=1-1]
	\arrow[""{name=1, anchor=center, inner sep=0}, "{U_C}", shift left=2, from=1-1, to=1-2]
	\arrow["\dashv"{anchor=center, rotate=-90}, draw=none, from=1, to=0]
\end{tikzcd}\]

which is exactly the forgetful-cofree adjunction for comodules.
\end{construction}

Next we move to show that $U_C$ is comonadic and preserves finite limits. First we need a short lemma:

\begin{lemma}\label{lemma:fiber_limit}
    Let $K$ be a weakly contractible diagram, and let $f: \D \to \C$ be a morphism of categories that preserves $K$-limits, and let $c \in \C$ be an object. Then fiber map $D_c \to \D$ creates $K$-limits. 
\end{lemma}

\begin{proof}
    Let's be given a diagram $p: K \to \D_c$. Then we can take the composition $p': K \to \D$, and suppose that we have a limit $L$ to $p': K \to \D$. Notice that since $f$ preserves $K$-limits, $fL$ must be sent to the limit of $fp': K \to \C$, which is just $c$ since $fp'$ is the constant diagram at $c$! Here we use the weak contractibility of $K$ to deduce that the limit is just $c$ again. 

    Hence we note that $fL \simeq c$. So $L$ lifts to an object in $\D_c$, and is clearly the totalization of $p$ in $D_c$.
\end{proof}

\begin{proposition}[Comonadicity of $U_C$]
The forgetful functor $U_C\Comod_C \to \C$ constructed in \ref{constr:comod_cofree} is comonadic.
\end{proposition}
\begin{proof}
We know that $U$ is comonadic. This gives that $U_C$ is conservative. We now need to show that it also preserves $U_C$-split totalizations, and we do so by first showing that $\Comod_C \to \Comod_Q$ preserves such totalizations. 

%Make this into a lemma
Notice that $p$ preserves all limits and colimits. Hence by lemma \ref{lemma:fiber_limit} we know that the map $\Comod_C \to \Comod_Q$ preserves and creates totalizations. The same is true for the map $\C \to \Coalg_Q \times \C$ by including at the coalgebra $C$. 

Now given a $U_C$-split totalization $X^\bullet$, we notice that $U_C X^\bullet$ is split. Hence including this into $\Coalg_Q \times \C$, it's still a split totalization. Now including $X^\bullet$ into $\Comod_Q$, we notice that $U X^\bullet$ is split by above, as its just the inclusion of $U_C X^\bullet$. Hence $U$ preserves this totalization, and thus so does $U_C$ because the inclusions create totalizations.
\end{proof}

Lastly we check that $U_C$ preserves limits. 
\begin{proposition}
    The functor $U_C: \Comod_C \to \C$ constructed in \ref{constr:comod_cofree} preserves limits.
\end{proposition}

\begin{proof}
We notice the following. Given a limit diagram $K \to \Comod_C$, we can represent each comodule object $(C, V_k)$ in this diagram for $k \in K$ by the underlying cosimplicial object
\[\begin{tikzcd}
	{(C, V_k)} & {\Md_Q \circ (C, V_k)} & {\Md^2_Q \circ (C, V_k)} & \dots & {\Md^n_Q \circ (C, V_k)} & \dots
	\arrow[from=1-1, to=1-2]
	\arrow[shift right, from=1-2, to=1-3]
	\arrow[shift left, from=1-2, to=1-3]
	\arrow[shift left=2, from=1-3, to=1-4]
	\arrow[shift right=2, from=1-3, to=1-4]
	\arrow[from=1-3, to=1-4]
	\arrow[""{name=0, anchor=center, inner sep=0}, shift right=3, from=1-4, to=1-5]
	\arrow[""{name=1, anchor=center, inner sep=0}, shift left=3, from=1-4, to=1-5]
	\arrow[""{name=2, anchor=center, inner sep=0}, shift left=3, from=1-5, to=1-6]
	\arrow[""{name=3, anchor=center, inner sep=0}, shift right=3, from=1-5, to=1-6]
	\arrow[shorten <=2pt, shorten >=2pt, phantom, "\dots"{rotate=90}, from=1, to=0]
	\arrow[shorten <=2pt, shorten >=2pt, phantom, "\dots"{rotate=90}, from=2, to=3]
\end{tikzcd}\]
in $\Fun(A \amalg dA, \C)$. 

Notice that $\Md_Q\circ (C, V)$ preserves limits in $V$, as its linear in $V$. Hence we can take the pointwise limit in $\Fun(A \amalg dA, \C)$ and get

\[\begin{tikzcd}
	{(C, \lim_K V_k)} & {\Md_Q \circ (C, \lim_K V_k)} & {\Md^2_Q \circ (C, \lim_K V_k)} & \dots & {\Md^n_Q \circ (C, \lim_K V_k)} & \dots
	\arrow[from=1-1, to=1-2]
	\arrow[shift right, from=1-2, to=1-3]
	\arrow[shift left, from=1-2, to=1-3]
	\arrow[shift left=2, from=1-3, to=1-4]
	\arrow[shift right=2, from=1-3, to=1-4]
	\arrow[from=1-3, to=1-4]
	\arrow[""{name=0, anchor=center, inner sep=0}, shift right=3, from=1-4, to=1-5]
	\arrow[""{name=1, anchor=center, inner sep=0}, shift left=3, from=1-4, to=1-5]
	\arrow[""{name=2, anchor=center, inner sep=0}, shift left=3, from=1-5, to=1-6]
	\arrow[""{name=3, anchor=center, inner sep=0}, shift right=3, from=1-5, to=1-6]
	\arrow[shorten <=2pt, shorten >=2pt, phantom, "\dots"{rotate=90}, from=1, to=0]
	\arrow[shorten <=2pt, shorten >=2pt, phantom, "\dots"{rotate=90}, from=2, to=3]
\end{tikzcd}\]
which gives $(C, \lim_k V_k)$ a natural $C$-comodule structure! Hence it lifts naturally to $\Comod_C$. With its induced $C$-comodule structure, $(C, \lim_k V_k)$ must clearly be the limit of the diagram $K \to \Comod_C$. Next we see that $U_C (C, \lim_k V_k) \simeq \lim_k V_k$ while $U_C (C, V_k) = V_k$ in $\C$. One can see this by using how $U: \Comod_Q \to \Coalg_Q \times \C$ is computed, then taking the fiber at $C \in \Coalg_Q$.

So we see that $U_C$ sends this limit $(C, \lim_k V_k)$ to the limit $\lim_k V_k$ in $\C$! Hence we see that $U_C$ preserves limits.
\end{proof}

So finally we have:
\begin{theorem}
    $\Comod_C$ as constructed in \ref{constr:comod_cofree} is a stable category.
\end{theorem}
\begin{proof}
    We know that $\Comod_C$ is comonadic over $\C$ via the forgetful-cofree adjunction, and that the comonad, which we'll denote by $M_C$, preserves limits by our propositions above. Hence the coalgebras of $M_C$ will be a stable category of $\C$ is stable, which means that $\Comod_C$ is stable.
\end{proof}

\section{Comodules and tangent complex}
The cotangent complex is a very well known construction. Historically it arose as an analog to sections of the tangent space to a manifold, just as rings are functions on a manifold. Hence, one can define for a given ring map $A \to B$, the set of Kahler differentials on $B$ over $A$, denoted as $\Omega_{B/A}^1$, with the following universal property:
\[
\Hom_B(\Omega_{B/A}^1, M) \simeq \operatorname{Der}_A(B, M).
\]
In other words, mapping from $\Omega_{B/A}^1$ to another $B$-module $M$ is the same as giving an $A$-linear derivation $B \to M$. This can be rephrased by using the following construction on $(B, M)$: we can make a new ring $B \oplus M$ whose multiplication is square-zero on $M$. Then an $A$-linear derivation $B \to M$ is equivalent to an $A$-algebra morphism from $B \to B \oplus M$.

Then this construction was given a homotopical flavor by using derived algebraic geometry, forming the cotangent complex. Historically this developed from the work of Andr\'e, Quillen, and Illusie \cite{Andre, Quillen, Illusie}. This approach was enveloped into higher algebra using higher categorical techniques by Lurie \cite[Section~7.3]{HA}.

Here we discuss a dual version of the cotangent complex for coalgebras instead, which we call the {\em tangent complex}. First, given a $Q$-coalgebra $C$ along with a $C$-comodule $M$, we construct a new $Q$-coalgebra $C \oplus M$. This coalgebra has square zero comultiplication on $M$, analogous to the square-zero construction of $B \oplus M$ for rings. Then we can define the tangent complex as the right adjoint to this $(C, M) \mapsto C \oplus M$ functor.

Next we move on to use the tangent complex functor and its adjoint to prove a key result: that comodules are the costabilization of coalgebras. This is again analogous to the algebraic result that modules are the stabilization of algebras, for example see \cite{FrancisTangent}.

\subsection{Construction of tangent complex}
\begin{construction}[Tangent, square-zero adjunction] \label{constr:tangent}

Fix a set $A$, $\V \in \CAlg(\Prlst)$, and $\C \in \CAlg(\PrlstV)$ following \ref{notation:AVC2}. Let $Q$ be a cooperad over $\V$. We construct an adjunction between $\Coalg_Q$ and $\Comod_Q$ fitting into the following diagram:

% https://q.uiver.app/#q=WzAsNCxbMCwwLCJcXENvbW9kX1EiXSxbMCwxLCJcXEZ1bihBXFxhbWFsZyBkQSwgXFxDKSJdLFsxLDEsIlxcRnVuKEEsIFxcQykiXSxbMSwwLCJcXENvYWxnX1EiXSxbMywyLCJVIiwyLHsib2Zmc2V0IjoyfV0sWzAsMywiXFxvcGx1cyIsMCx7Im9mZnNldCI6LTJ9XSxbMywwLCJUIiwwLHsib2Zmc2V0IjotMn1dLFsxLDAsIlxcQ29mcmVlX3tcXE1kX1F9IiwyLHsib2Zmc2V0IjoyfV0sWzAsMSwiVSIsMix7Im9mZnNldCI6Mn1dLFsyLDMsIlxcQ29mcmVlX1EiLDIseyJvZmZzZXQiOjJ9XSxbMSwyLCJcXG9wbHVzIiwwLHsib2Zmc2V0IjotMn1dLFsyLDEsIlxcRGVsdGEiLDAseyJvZmZzZXQiOi0yfV0sWzUsNiwiIiwwLHsibGV2ZWwiOjEsInN0eWxlIjp7Im5hbWUiOiJhZGp1bmN0aW9uIn19XSxbNCw5LCIiLDAseyJsZXZlbCI6MSwic3R5bGUiOnsibmFtZSI6ImFkanVuY3Rpb24ifX1dLFs4LDcsIiIsMix7ImxldmVsIjoxLCJzdHlsZSI6eyJuYW1lIjoiYWRqdW5jdGlvbiJ9fV0sWzEwLDExLCIiLDAseyJsZXZlbCI6MSwic3R5bGUiOnsibmFtZSI6ImFkanVuY3Rpb24ifX1dXQ==
\[\begin{tikzcd}
	{\Comod_Q} & {\Coalg_Q} \\
	{\Fun(A\amalg dA, \C)} & {\Fun(A, \C)}
	\arrow[""{name=0, anchor=center, inner sep=0}, "U"', shift right=2, from=1-2, to=2-2]
	\arrow[""{name=1, anchor=center, inner sep=0}, "\oplus", shift left=2, from=1-1, to=1-2]
	\arrow[""{name=2, anchor=center, inner sep=0}, "T", shift left=2, from=1-2, to=1-1]
	\arrow[""{name=3, anchor=center, inner sep=0}, "{\Cofree_{\Md_Q}}"', shift right=2, from=2-1, to=1-1]
	\arrow[""{name=4, anchor=center, inner sep=0}, "U"', shift right=2, from=1-1, to=2-1]
	\arrow[""{name=5, anchor=center, inner sep=0}, "{\Cofree_Q}"', shift right=2, from=2-2, to=1-2]
	\arrow[""{name=6, anchor=center, inner sep=0}, "\oplus", shift left=2, from=2-1, to=2-2]
	\arrow[""{name=7, anchor=center, inner sep=0}, "\Delta", shift left=2, from=2-2, to=2-1]
	\arrow["\dashv"{anchor=center, rotate=-90}, draw=none, from=1, to=2]
	\arrow["\dashv"{anchor=center}, draw=none, from=0, to=5]
	\arrow["\dashv"{anchor=center}, draw=none, from=4, to=3]
	\arrow["\dashv"{anchor=center, rotate=-90}, draw=none, from=6, to=7]
\end{tikzcd}\]
where $\oplus: \Fun(A \amalg dA, \C) \to \Fun(A, \C)$ sends $F$ to $a \mapsto F(a) \oplus F(da)$, and $\Delta: \Fun(A, \C)\to \Fun(A \amalg dA, \C)$ extends $F$ to $dA$ via $da \mapsto F(a)$. 

We first construct the left adjoint $\oplus: \Comod_Q \to \Coalg_Q$. For this it's enough to show that $\oplus: \Fun(A \amalg dA, \C) \to \Fun(A, \C)$ is a $\Sseq_A$-colax map \ref{def:lax} from the action of $\Sseq_A$ (here $\Sseq_A$ acts on $\Fun(A \amalg dA, \C)$ using $\Md: \Sseq_A \to \Sseq_{A \amalg dA}$). 

Hence given $S \in \Sseq_A$, we want natural maps $\eta_S$ fitting into
% https://q.uiver.app/#q=WzAsNCxbMCwwLCJcXEZ1bihBIFxcYW1hbGcgZEEsIFxcQykiXSxbMCwxLCJcXEZ1bihBLCBcXEMpIl0sWzEsMCwiXFxGdW4oQSBcXGFtYWxnIGRBLCBcXEMpIl0sWzEsMSwiXFxGdW4oQSwgXFxDKSJdLFswLDEsIlxcb3BsdXMiXSxbMCwyLCJcXE1kX1MgXFxjaXJjIiwyXSxbMiwzLCJcXG9wbHVzIiwyXSxbMSwzLCJTIFxcY2lyYyJdLFszLDBdXQ==
\[\begin{tikzcd}
	{\Fun(A \amalg dA, \C)} & {\Fun(A \amalg dA, \C)} \\
	{\Fun(A, \C)} & {\Fun(A, \C)}
	\arrow["\oplus", from=1-1, to=2-1]
	\arrow["{\Md_S \circ {-}}", from=1-1, to=1-2]
	\arrow["\oplus", from=1-2, to=2-2]
	\arrow["{S \circ {-}}", from=2-1, to=2-2]
	\arrow[Rightarrow, "\eta_S", from=1-2, to=2-1]
\end{tikzcd}\]
and commuting with the $\circ$ action.

However this is easy to see because given $S \in \Sseq_A$ and $(C_a, M_{da})_{a \in A}$ in $\Fun(A \amalg dA, \C)$, we just take $\eta_S(C_a, M_{da}): \oplus(\Md_S \circ (C_i, M_i)) \to S \circ (C_a \oplus M_{da})$ to be the the direct summand inclusion consisting of the trees with at most one leaf using an $M_{da}$, which is exactly how $\Md_S$ is defined. In more detail, note that $S \circ (C_a \oplus M_{da})$ consists of terms 
\[
    \bigoplus_{\{a_1, \dots, a_m\}} S(\{a_1, \dots, a_m\}, a) \otimes_{\Sigma_m} [(C_{a_1} \oplus M_{da_1}) \otimes \dots \otimes (C_{a_m} \oplus M_{da_m})],  
\]
which after distributing the direct sums, consists of terms with $S(\{a_1, \dots, a_m\}, a)$ tensored with strings of $C_{a_i}$ and $M_{da_j}$. Then note that $\oplus(\Md_S \circ (C_i, M_i))$ consists of direct sums of the same strings except that one can only allow at most one $\M_{da_j}$. The colax naturality is clear from our definitions as direct sum inclusions. 

Hence we have an induced map on coalgebras:
\[\begin{tikzcd}
	{\Comod_Q} & {\Fun(A\amalg dA, \C)} \\
	{\Coalg_Q} & {\Fun(A, \C),}
	\arrow[""{name=0, anchor=center, inner sep=0}, "\oplus", from=1-2, to=2-2]
	\arrow[""{name=1, anchor=center, inner sep=0}, "U", from=2-1, to=2-2]
	\arrow[""{name=2, anchor=center, inner sep=0}, "\oplus", from=1-1, to=2-1]
	\arrow[""{name=5, anchor=center, inner sep=0}, "U", from=1-1, to=1-2]
\end{tikzcd}\]

Now clearly $\oplus$ preserves colimits (for example it follows since $U$ creates colimits, and on underlying categories $\oplus$ preserves colimits). Hence it has a right adjoint, and the right adjoint commutes because the left adjoints do. Hence we have our full diagram
% https://q.uiver.app/#q=WzAsNCxbMCwwLCJcXENvbW9kX1EiXSxbMCwxLCJcXEZ1bihBXFxhbWFsZyBkQSwgXFxDKSJdLFsxLDEsIlxcRnVuKEEsIFxcQykiXSxbMSwwLCJcXENvYWxnX1EiXSxbMywyLCJVIiwyLHsib2Zmc2V0IjoyfV0sWzAsMywiXFxvcGx1cyIsMCx7Im9mZnNldCI6LTJ9XSxbMywwLCJUIiwwLHsib2Zmc2V0IjotMn1dLFsxLDAsIlxcQ29mcmVlX3tcXE1kX1F9IiwyLHsib2Zmc2V0IjoyfV0sWzAsMSwiVSIsMix7Im9mZnNldCI6Mn1dLFsyLDMsIlxcQ29mcmVlX1EiLDIseyJvZmZzZXQiOjJ9XSxbMSwyLCJcXG9wbHVzIiwwLHsib2Zmc2V0IjotMn1dLFsyLDEsIlxcRGVsdGEiLDAseyJvZmZzZXQiOi0yfV0sWzUsNiwiIiwwLHsibGV2ZWwiOjEsInN0eWxlIjp7Im5hbWUiOiJhZGp1bmN0aW9uIn19XSxbNCw5LCIiLDAseyJsZXZlbCI6MSwic3R5bGUiOnsibmFtZSI6ImFkanVuY3Rpb24ifX1dLFs4LDcsIiIsMix7ImxldmVsIjoxLCJzdHlsZSI6eyJuYW1lIjoiYWRqdW5jdGlvbiJ9fV0sWzEwLDExLCIiLDAseyJsZXZlbCI6MSwic3R5bGUiOnsibmFtZSI6ImFkanVuY3Rpb24ifX1dXQ==
\[\begin{tikzcd}
	{\Comod_Q} & {\Coalg_Q} \\
	{\Fun(A\amalg dA, \C)} & {\Fun(A, \C)}
	\arrow[""{name=0, anchor=center, inner sep=0}, "U"', shift right=2, from=1-2, to=2-2]
	\arrow[""{name=1, anchor=center, inner sep=0}, "\oplus", shift left=2, from=1-1, to=1-2]
	\arrow[""{name=2, anchor=center, inner sep=0}, "T", shift left=2, from=1-2, to=1-1]
	\arrow[""{name=3, anchor=center, inner sep=0}, "{\Cofree_{\Md_Q}}"', shift right=2, from=2-1, to=1-1]
	\arrow[""{name=4, anchor=center, inner sep=0}, "U"', shift right=2, from=1-1, to=2-1]
	\arrow[""{name=5, anchor=center, inner sep=0}, "{\Cofree_Q}"', shift right=2, from=2-2, to=1-2]
	\arrow[""{name=6, anchor=center, inner sep=0}, "\oplus", shift left=2, from=2-1, to=2-2]
	\arrow[""{name=7, anchor=center, inner sep=0}, "\Delta", shift left=2, from=2-2, to=2-1]
	\arrow["\dashv"{anchor=center, rotate=-90}, draw=none, from=1, to=2]
	\arrow["\dashv"{anchor=center}, draw=none, from=0, to=5]
	\arrow["\dashv"{anchor=center}, draw=none, from=4, to=3]
	\arrow["\dashv"{anchor=center, rotate=-90}, draw=none, from=6, to=7]
\end{tikzcd}\]
\end{construction}
%NEED AN ADJUNCTION FROM COMOD TO ARROWS OF COALG. 

\begin{remark}
    Notice that intuitively, the left adjoint $\oplus: \Comod_Q \to \Coalg_Q$ sends $(C, M)$ to the square zero coalgebra $C \oplus M$ where $M$ has comultiplication zero. 
\end{remark}

A fact that we'll use in showing that comodules are the costabilization of coalgebras is the following:
\begin{construction}
    The left adjoint $\oplus: \Comod_Q \to \Coalg_Q$ constructed in \ref{constr:tangent} is augmented by the fibration $p:\Comod_Q \to \Coalg_Q$ as constructed in \ref{constr:modfib} which lies over the inclusion and quotient augmentation of $\oplus$ by $\pi_1$ as functors $\Fun(A \amalg dA, \C) \to \Fun(A, \C).$ In other words, we have the following diagram in functor categories
    % https://q.uiver.app/#q=WzAsMyxbMCwwLCJwIl0sWzEsMSwiXFxvcGx1cyJdLFswLDIsInAiXSxbMCwxXSxbMSwyXSxbMCwyLCIxX3AiLDJdXQ==
\[\begin{tikzcd}
	p \\
	& \oplus \\
	p
	\arrow[from=1-1, to=2-2]
	\arrow[from=2-2, to=3-1]
	\arrow["{1_p}"', from=1-1, to=3-1]
\end{tikzcd}\]
\end{construction}
This just records the fact that the square zero coalgebra $C \oplus M$ fits into the diagram
% https://q.uiver.app/#q=WzAsMyxbMCwwLCJDIl0sWzEsMSwiQyBcXG9wbHVzIE0iXSxbMCwyLCJDIl0sWzAsMV0sWzEsMl0sWzAsMiwiMV9DIiwyXV0=
\[\begin{tikzcd}
	C \\
	& {C \oplus M} \\
	C
	\arrow[from=1-1, to=2-2]
	\arrow[from=2-2, to=3-1]
	\arrow["{1_C}"', from=1-1, to=3-1]
\end{tikzcd}\]

\begin{proof}
    This is an easy consequence of definitions. The fibration $p$ is defined via a strong $\Sseq_A$-monoidal structure on $\pi_1: \Fun(A \amalg dA, \C) \to \Fun(A, \C)$. The map $\oplus$ is defined via a colax $\Sseq_A$-monoidal structure on the undecorated $\oplus: \Fun(A \amalg dA, \C) \to \Fun(A, \C)$ \ref{def:lax}. 
    
    Notice that this reduced $\oplus: \Fun(A \amalg dA, \C) \to \Fun(A, \C)$ is naturally augmented by $\pi_1$: there's a natural inclusion transformation $\pi_1 \to \oplus$ whose components are $i_C: C \to \C \oplus M$. There's also a natural quotient transformation $\oplus \to \pi_1$ whose components are $q_C: C\oplus M \to C$. 

    One can easily check that these two natural transformations commute with the colax $\Sseq_A$-monoidal structures \ref{def:lax}. For example, given $C, M \in \C$ and $S \in \Sseq_A$, the natural inclusion satisfies
    % https://q.uiver.app/#q=WzAsNSxbMCwwLCJRXFxjaXJjIEMiXSxbMCwxLCJRIFxcY2lyYyBDIFxcb3BsdXMgUSBcXGNpcmMoQztNKSJdLFsyLDAsIlFcXGNpcmMgQyJdLFsyLDEsIlEgXFxjaXJjIChDIFxcb3BsdXMgTSkiXSxbMSwwXSxbMCwxLCJpX3tRXFxjaXJjIEN9Il0sWzAsMiwiMV97USBcXGNpcmMgQ30gPSBjX3tcXHBpXzF9IiwyXSxbMiwzLCJRXFxjaXJjIGlfQyJdLFsxLDMsImNfcCIsMl1d
\[\begin{tikzcd}
	{S\circ C} & {} & {S\circ C} \\
	{S \circ C \oplus S \circ(C;M)} && {S \circ (C \oplus M)},
	\arrow["{i_{S\circ C}}", from=1-1, to=2-1]
	\arrow["{1_{S \circ C} = c_{\pi_1}}"', from=1-1, to=1-3]
	\arrow["{S\circ i_C}", from=1-3, to=2-3]
	\arrow["{c_p}"', from=2-1, to=2-3]
\end{tikzcd}\]
    where $c_{-}$ stands for the structure morphism for colaxness of either $\pi_1$ or $p$. The square commutes because both legs compute the natural inclusion of $S \circ C$ into $S \circ (C \oplus M)$. The higher coherences for the colax morphisms are also such natural inclusions. Similarly, the quotient satisfies a dual diagram, and clearly commutes with the colaxness morphisms.

    Hence the natural inclusion transformation and natural quotient transformation lift uniquely to natural transformations $p \to \oplus$ and $\oplus \to p$, which compose to the identity, as required.
\end{proof}

\begin{remark}\label{remark:aug_oplus}
    Now we notice that for each coalgebra $C$, the above result shows that the $\oplus$ functor can be thought of as giving a morphism $\Comod_C \to \Coalg_Q^{C/-/C}$. In other words, codomain is $C$-coaugmented $Q$-coalgebras.
\end{remark}

Back to tangent complexes. Once we've defined tangent complexes, we can also discuss derivations.
\begin{definition}
    Fix a set $A$, $\V \in \CAlg(\Prlst)$, and $\C \in \CAlg(\PrlstV)$ following \ref{notation:AVC2}. Let $Q$ be a cooperad over $\V$. Given a comodule $M \in \Comod_C$ and a morphism $f: C \to D$, a {\em $C$-derivation from $M$ to $D$}  is the structure of a diagram 
    \[
    \begin{tikzcd}
        C \oplus M \ar[r, "\xi"] & D \\ 
        & C \ar[u, "f"] \ar[ul]
    \end{tikzcd}
    \]
    in $\Coalg_Q$, where the morphism $C \to C \oplus M$ is the natural inclusion. 

    Notice that this data is equivalent to a morphism $\xi': (C, M) \to (D, T_D)$ in $\Comod_Q$ lying over $f: C \to D$. When the context is clear, we sometimes denote such a derivation by just $\xi: M \to D$ in $\Fun(A, \C)$. Here $M$ is thought of as a subobject of $C \oplus M$ in $\Fun(A, \C)$. 
\end{definition}

\begin{definition}\label{def:univ_coder}
    We note that the counit of the adjunction $\oplus \dashv T$ is a morphism
    \[
    C \oplus T_C \to C
    \]
    in $\Coalg_Q$. We call this the {\em universal derivation $\xi_C: T_C \to C$}.
\end{definition}

\begin{example}[Tangent complex of a cofree coalgebra]\label{example:tangent_cofree}
Given $C := \Cofree(V)$ a cofree coalgebra. Then the tangent complex $T_{\Cofree(V)}$ can be calculated by as follows. Notice that we have the diagram 
% https://q.uiver.app/#q=WzAsNCxbMCwwLCJcXENvbW9kX1EiXSxbMSwwLCJcXENvYWxnX1EiXSxbMSwxLCJcXE1vZF9rIl0sWzAsMSwiXFxNb2ReMl9rIl0sWzEsMCwiVCIsMCx7Im9mZnNldCI6LTJ9XSxbMCwxLCJcXG9wbHVzIiwwLHsib2Zmc2V0IjotMn1dLFsxLDIsIlUiLDIseyJvZmZzZXQiOjJ9XSxbMCwzLCJVIiwyLHsib2Zmc2V0IjoyfV0sWzMsMiwiXFxvcGx1cyIsMCx7Im9mZnNldCI6LTJ9XSxbMiwzLCJcXERlbHRhIiwwLHsib2Zmc2V0IjotMn1dLFszLDAsIlxcQ29mcmVlIiwyLHsib2Zmc2V0IjoyfV0sWzIsMSwiXFxDb2ZyZWUiLDIseyJvZmZzZXQiOjJ9XSxbNSw0LCIiLDAseyJsZXZlbCI6MSwic3R5bGUiOnsibmFtZSI6ImFkanVuY3Rpb24ifX1dLFs4LDksIiIsMCx7ImxldmVsIjoxLCJzdHlsZSI6eyJuYW1lIjoiYWRqdW5jdGlvbiJ9fV0sWzcsMTAsIiIsMix7ImxldmVsIjoxLCJzdHlsZSI6eyJuYW1lIjoiYWRqdW5jdGlvbiJ9fV0sWzYsMTEsIiIsMix7ImxldmVsIjoxLCJzdHlsZSI6eyJuYW1lIjoiYWRqdW5jdGlvbiJ9fV1d
\[\begin{tikzcd}
	{\Comod_Q} & {\Coalg_Q} \\
	{\Fun(A \amalg dA, \C)} & {\Fun(A, \C)}
	\arrow[""{name=0, anchor=center, inner sep=0}, "T", shift left=2, from=1-2, to=1-1]
	\arrow[""{name=1, anchor=center, inner sep=0}, "\oplus", shift left=2, from=1-1, to=1-2]
	\arrow[""{name=2, anchor=center, inner sep=0}, "U"', shift right=2, from=1-2, to=2-2]
	\arrow[""{name=3, anchor=center, inner sep=0}, "U"', shift right=2, from=1-1, to=2-1]
	\arrow[""{name=4, anchor=center, inner sep=0}, "\oplus", shift left=2, from=2-1, to=2-2]
	\arrow[""{name=5, anchor=center, inner sep=0}, "\Delta", shift left=2, from=2-2, to=2-1]
	\arrow[""{name=6, anchor=center, inner sep=0}, "\Cofree"', shift right=2, from=2-1, to=1-1]
	\arrow[""{name=7, anchor=center, inner sep=0}, "\Cofree"', shift right=2, from=2-2, to=1-2]
	\arrow["\dashv"{anchor=center, rotate=-90}, draw=none, from=1, to=0]
	\arrow["\dashv"{anchor=center, rotate=-90}, draw=none, from=4, to=5]
	\arrow["\dashv"{anchor=center}, draw=none, from=3, to=6]
	\arrow["\dashv"{anchor=center}, draw=none, from=2, to=7]
\end{tikzcd}\]
which lets us compute $T_{\Cofree(V)}$ by instead using $\Cofree_{M_Q}(\Delta(V))$. Here $\Delta$ is the diagonal functor, the right adjoint to $\oplus$. Calculating we see that the result is $Q \circ (V; V)$, or $\Cofree_{M_Q}(V, V)$. 
\end{example}

\begin{definition}[Relative tangent complex]
    Given a map $f:C \to C'$ of $Q$-coalgebras, we can define the {\em relative tangent complex} $T_f$ (or $T_{C/C'}$ if the map is clear from context) as the fiber of 
    \[
    T_f \to T_C \to T_{C'}
    \]
    in $\Comod_Q$. This is equivalent to the fiber of $T_C \to f^{\ast} T_{C'}$ in $\Comod_C$. 
\end{definition}

%CHECK THIS IN THE MULTI-OBJECT CASE
\begin{example}[Relative tangent complex of cofree coalgebras]\label{example:tangent_cofree_map}
    Given $f: V \to W$, we calculate $T_{\Cofree f}$. 

    By the previous example \ref{example:tangent_cofree}, we see that the relative tangent must be 
    \[
    \fib(Q \circ (V; V) \to (\Cofree{f})^{\ast} Q \circ (W;W)).
    \]
    Notice that $(\Cofree{f})^{\ast} Q \circ (W;W)$ is equivalent to $Q \circ (V; W)$. Now since the $Q \circ (V; -)$ construction is linear, we can easily calculate this as 
    \[
    Q \circ(V; \fib(f)).
    \]
\end{example}

\begin{remark}\label{remark:rel_tangent_adjoint_to_aug_oplus}
    In the situation of \ref{remark:aug_oplus}, we notice that the right adjoint to 
    \[
    \oplus: \Comod_C \to \Coalg_Q^{C/-/C}
    \] 
    is given by the relative tangent complex $T_{-/C}$ which is indeed a $C$-comodule.
\end{remark}

\subsection{Forgetful-cofree adjunction between coalgebras and comodules}
In this section we construct the category of augmented coalgebras. Intuitively, it consists of a pair of coalgebras $C, D$ where $D$ is a $C$-augmented coalgebra. Hence the fibers of this category will be $\Coalg_Q^{C/-/C}$. This category is very important for comodules as we shall see that $\Comod_Q$ is exactly the costabilization of $\Coalg_Q^{C/-/C}$. 

We move on to construct an augmented forgetful functor \[U_C^{\aug}: \Coalg_Q^{C/-/C} \to \Comod_C.\] 

Intuitively, if we had a natural symmetric monoidal product on $\Comod_C$, we should have \[
\Coalg_Q^{C/-/C} \simeq \Coalg^{\nonunit}_Q(\Comod_C),
\] where $\nonunit$ stands for nonunital coalgebras, hence one removes the $0$-ary operations in $Q$.
Then the forgetful functor would just be the forgetful functor of a $Q$-coalgebra structure. However since we don't have an easy symmetric monoidal structure to use on $\Comod_C$, we must construct the forgetful functor directly by hand.

Now in order to define augmented coalgebras, we first define augmented objects in an arbitrary category.
\begin{notation}
    We will use $\eqTriangle$ to denote the $2$-simplex $\Delta^2$ with a degenrate arrow $[0,2]$. In other words, it represents the diagram
    % https://q.uiver.app/#q=WzAsMyxbMCwxLCIwIl0sWzIsMSwiMiJdLFsxLDAsIjEiXSxbMCwxLCI9IiwyXSxbMCwyXSxbMiwxXV0=
\[\begin{tikzcd}
	& 1 \\
	0 && 2.
	\arrow["{=}"', from=2-1, to=2-3]
	\arrow[from=2-1, to=1-2]
	\arrow[from=1-2, to=2-3]
\end{tikzcd}\]
This diagram is exactly the walking augmented object: hence an augmented object $c \to d \to c$ in $\C$ is exactly the same as a map $\eqTriangle \to \C$.
\end{notation}

\begin{construction}[Augmentation map]
Define map $\aug: \C^{\eqTriangle} \to \C \times \C$ sending the triangle 
% https://q.uiver.app/#q=WzAsMyxbMCwwLCJWIl0sWzEsMSwiVyJdLFswLDIsIlYiXSxbMCwxLCJmIl0sWzEsMiwiZyJdLFswLDIsIjFfViIsMl1d
\[\begin{tikzcd}
	V \\
	& W \\
	V
	\arrow["f", from=1-1, to=2-2]
	\arrow["g", from=2-2, to=3-1]
	\arrow["{1_V}"', from=1-1, to=3-1]
\end{tikzcd}\]
to the pair $(V, \cofib(f)) \simeq (V, \fib(g))$.
We can use $\aug$ pointwise to get $\aug: \Fun(A, \C)^{\eqTriangle} \to \Fun(A \amalg dA, \C)$, where $A$ is a set.
\end{construction}

Now we prove a simple lemma that says the data of an augmented object $c \to d \to c$ in a stable category $\C$ is the same as the pair of augmenting object $c$ and the cofiber of $c \to d$ (or equivalently the fiber of $d \to c$).

\begin{lemma}
    Let $\C$ be a stable category. Then the map $\aug: \C^{\eqTriangle} \to \C \times \C$ is an equivalence. Thus, so is the pointwise version $\aug: \Fun(A, \C)^{\eqTriangle} \to \Fun(A \amalg dA, \C)$ where $A$ is a set. 
\end{lemma}
\begin{proof}
Notice that augmented triangles
\[\begin{tikzcd}
	V \\
	& W \\
	V
	\arrow["f", from=1-1, to=2-2]
	\arrow["g", from=2-2, to=3-1]
	\arrow["{1_V}"', from=1-1, to=3-1]
\end{tikzcd}\]
in $\C$ give natural splittings of $W$ as the direct sum $V \oplus \cofib(f)$. Hence we have a natural inverse functor for $\aug$ sending a pair $(V, M)$ to the triangle 
\[
\begin{tikzcd}
	V \\
	& V \oplus M \\
	V.
	\arrow["i_V", from=1-1, to=2-2]
	\arrow["q_V", from=2-2, to=3-1]
	\arrow["{1_V}"', from=1-1, to=3-1]
\end{tikzcd}
\] 
\end{proof}

Now we can construct the augmented forgetful functor from $C$-augmented $Q$-coalgebras to $C$-comodules. First we start with the global version on global comodules.

\begin{construction}[Augmented forgetful functor]\label{constr:U_aug}
Fix a set $A$, $\V \in \CAlg(\Prlst)$, and $\C \in \CAlg(\PrlstV)$ following \ref{notation:AVC2}. Let $Q$ be a cooperad over $\V$. We construct a functor $\Coalg_Q^{\eqTriangle} \to \Comod_Q$ that lifts the following adjunction:
% https://q.uiver.app/#q=WzAsNCxbMCwxLCJcXEZ1bihBLCBcXEMpXntcXGVxVHJpYW5nbGV9Il0sWzAsMCwiXFxDb2FsZ19RXntcXGVxVHJpYW5nbGV9Il0sWzEsMCwiXFxDb21vZF9RIl0sWzEsMSwiXFxGdW4oQSBcXGFtYWxnIGRBLCBcXEMpIl0sWzEsMCwiVSIsMix7Im9mZnNldCI6Mn1dLFswLDEsIlxcQ29mcmVlX1EiLDIseyJvZmZzZXQiOjJ9XSxbMiwzLCJVX3tcXE1kX1F9IiwyLHsib2Zmc2V0IjoyfV0sWzMsMiwiXFxDb2ZyZWVfe1xcTWRfUX0iLDIseyJvZmZzZXQiOjJ9XSxbMCwzLCJcXGF1ZyIsMl0sWzEsMiwiXFxhdWciLDIseyJzdHlsZSI6eyJib2R5Ijp7Im5hbWUiOiJkYXNoZWQifX19XSxbNCw1LCIiLDIseyJsZXZlbCI6MSwic3R5bGUiOnsibmFtZSI6ImFkanVuY3Rpb24ifX1dLFs2LDcsIiIsMix7ImxldmVsIjoxLCJzdHlsZSI6eyJuYW1lIjoiYWRqdW5jdGlvbiJ9fV1d
\[\begin{tikzcd}
	{\Coalg_Q^{\eqTriangle}} & {\Comod_Q} \\
	{\Fun(A, \C)^{\eqTriangle}} & {\Fun(A \amalg dA, \C)}.
	\arrow[""{name=0, anchor=center, inner sep=0}, "U"', shift right=2, from=1-1, to=2-1]
	\arrow[""{name=1, anchor=center, inner sep=0}, "{\Cofree_Q}"', shift right=2, from=2-1, to=1-1]
	\arrow[""{name=2, anchor=center, inner sep=0}, "{U_{\Md_Q}}"', shift right=2, from=1-2, to=2-2]
	\arrow[""{name=3, anchor=center, inner sep=0}, "{\Cofree_{\Md_Q}}"', shift right=2, from=2-2, to=1-2]
	\arrow["\aug"', from=2-1, to=2-2]
	\arrow["U^{\aug}"', dashed, from=1-1, to=1-2]
	\arrow["\dashv"{anchor=center}, draw=none, from=0, to=1]
	\arrow["\dashv"{anchor=center}, draw=none, from=2, to=3]
\end{tikzcd}\]

To do so we once again check that the undecorated $\aug: {\Fun(A, \C)^{\eqTriangle}} \to {\Fun(A \amalg dA, \C)}$ is $\Sseq_A$-colax monoidal \ref{def:lax}. Notice however that if we write a triangle
\[\begin{tikzcd}
	V \\
	& W \\
	V
	\arrow["f", from=1-1, to=2-2]
	\arrow["g", from=2-2, to=3-1]
	\arrow["{1_V}"', from=1-1, to=3-1]
\end{tikzcd}\] as 
\[
\begin{tikzcd}
	V \\
	& V \oplus M \\
	V.
	\arrow["i_V", from=1-1, to=2-2]
	\arrow["q_V", from=2-2, to=3-1]
	\arrow["{1_V}"', from=1-1, to=3-1]
\end{tikzcd}
\] 
using the natural splitting, then there is a a natural morphism from 
\[
l^n_{(V, M)}: \cofib({V^n \to (V \oplus M)^n}) \to \bigoplus_{a + b + 1 = n} V^a \otimes M \otimes V^b,
\]
which only remembers the terms that are linear in $M$. If you apply $S(n) \otimes_{\Sigma_n} {-}$ to these morphisms, then you get the colax morphism 
\[
\aug (S \circ (V \to V \oplus M \to V)) := (S\circ V, \cofib(S \circ V \to S \circ (V \oplus M)) \to (S\circ V, S\circ (V; M) := \Md_S \circ (V, M).
\]
They respect the $\Sseq_A$-action clearly since the $\circ$ action doesn't disturb the maps $l^n_{(V, M)}$.
\end{construction}

Next we'd like to show $U^{\aug}$ has a right adjoint, and that this pair is a relative adjunction over $\Coalg_Q$.

\begin{proposition}
    The functor $U^{\aug}: \Coalg^{\eqTriangle}_Q \to \Comod_Q$ \ref{constr:U_aug} has a right adjoint which we call $\Cofree^{\aug}$. Let $q: \Coalg^{\eqTriangle}_Q \to \Coalg_Q$ be the funcotr that only remembers the coalgebra at $0 \in \eqTriangle$, ie the augmenting $Q$-coalgebra. Then this adjoint pair is comonadic, and furthermore is a relative adjunction over the $\Coalg_Q$: 
    % https://q.uiver.app/#q=WzAsMyxbMCwwLCJcXENvYWxnXntcXGVxVHJpYW5nbGV9X1EiXSxbMiwwLCJcXENvbW9kX1EiXSxbMSwxLCJcXENvYWxnX1EiXSxbMCwxLCJVXntcXGF1Z30iLDAseyJvZmZzZXQiOi0yfV0sWzEsMCwiXFxDb2ZyZWVee1xcYXVnfSIsMCx7Im9mZnNldCI6LTJ9XSxbMCwyLCJxIiwyXSxbMSwyLCJwIl0sWzMsNCwiIiwwLHsibGV2ZWwiOjEsInN0eWxlIjp7Im5hbWUiOiJhZGp1bmN0aW9uIn19XV0=
\[\begin{tikzcd}
	{\Coalg^{\eqTriangle}_Q} && {\Comod_Q} \\
	& {\Coalg_Q},
	\arrow[""{name=0, anchor=center, inner sep=0}, "{U^{\aug}}", shift left=2, from=1-1, to=1-3]
	\arrow[""{name=1, anchor=center, inner sep=0}, "{\Cofree^{\aug}}", shift left=2, from=1-3, to=1-1]
	\arrow["q"', from=1-1, to=2-2]
	\arrow["p", from=1-3, to=2-2]
	\arrow["\dashv"{anchor=center, rotate=-90}, draw=none, from=0, to=1]
\end{tikzcd}\]
where $p$ is constructed in \ref{constr:modfib}.
\end{proposition}

\begin{proof}
    First we start by showing that the right adjoint exists. It's enough then to show that $U^{\aug}$ preserves colimits as our categories are presentable. 
    
    Notice that in 
    \[\begin{tikzcd}
	{\Coalg_Q^{\eqTriangle}} & {\Comod_Q} \\
	{\Fun(A, \C)^{\eqTriangle}} & {\Fun(A \amalg dA, \C)},
	\arrow[""{name=0, anchor=center, inner sep=0}, "U"', shift right=2, from=1-1, to=2-1]
	\arrow[""{name=1, anchor=center, inner sep=0}, "{\Cofree_Q}"', shift right=2, from=2-1, to=1-1]
	\arrow[""{name=2, anchor=center, inner sep=0}, "{U_{\Md_Q}}"', shift right=2, from=1-2, to=2-2]
	\arrow[""{name=3, anchor=center, inner sep=0}, "{\Cofree_{\Md_Q}}"', shift right=2, from=2-2, to=1-2]
	\arrow["\aug"', from=2-1, to=2-2]
	\arrow["U^{\aug}"', dashed, from=1-1, to=1-2]
	\arrow["\dashv"{anchor=center}, draw=none, from=0, to=1]
	\arrow["\dashv"{anchor=center}, draw=none, from=2, to=3]
\end{tikzcd}\]
the functor $U$ and $U_{\Md_Q}$ both create colimits, ie preserve colimits and are conservative. Hence since $U_{\Md_Q} U^{\aug} \simeq \aug U$, and using the fact that $\aug$ is an equivalence, we see that $U^{\aug}$ also preserves colimits, and further is in fact comonadic (as $U$ and $U_{\Md_Q}$. 

Next we check that this adjunction is relative over $\Coalg_Q$. To do so, we look at the counit 
\[
\epsilon_{\aug}: U^{\aug} \Cofree^{\aug} \to 1_{\Comod_Q}
\]
that $p$ sends it to an equivalence. To do so, note in the map $p U^{\aug} = q$. This is because $p U^{\aug}$ lie over 
\[
\Fun(A, \C)^{\eqTriangle} \to \Fun(A \amalg dA, \C) \to \Fun(A, \C)\]
which exactly only remembers the augmenting object, ie sends $V \to W \to V$ to $V$. 

Since $q$ preserves all limits and colimits. So, when checking if 
\[
p \epsilon_{\aug}: p U^{\aug} \Cofree^{\aug} = q \Cofree^{\aug} \to p
\] is an equivalence, we can check after taking cofree resolutions of comodules. Further we can also apply $U_Q$ to this map as it is a conservative functor.

Hence, it's enough to check whether
\[
U_Q p \epsilon \Cofree_{\Md_Q}: U_Q p U^{\aug} \Cofree^{\aug} \Cofree{\Md_Q} \to U_Q p \Cofree{\Md_Q} 
\]
is an equivalence. This reduces to 
\[
\pi_1 (\aug(Q \circ (V \to V \oplus M \to V)) \to \Md_Q \circ (V, M))
\]
which is $\pi_1$ applied to the colax map for $\aug$. However, its projection $\pi_1$ on the first factor (or on the coaugmenting $Q$-coalgebra) is an equivalence! It is just the identity on $Q \circ V$. Hence we have a relative adjunction, as required.
\end{proof}

Now we have the augmented cofree adjunction:
\begin{corollary}
   Fix a set $A$, $\V \in \CAlg(\Prlst)$, and $\C \in \CAlg(\PrlstV)$ following \ref{notation:AVC2}. Let $Q$ be a cooperad over $\V$. Then we have for every $Q$-coalgebra $C$, an adjunction
% https://q.uiver.app/#q=WzAsMixbMCwwLCJcXENvYWxnX1Fee0MvLS9DfSJdLFsxLDAsIlxcQ29tb2RfQyJdLFswLDEsIlVee1xcYXVnfV9DIiwwLHsib2Zmc2V0IjotMn1dLFsxLDAsIlxcQ29mcmVlXntcXGF1Z31fQyIsMCx7Im9mZnNldCI6LTJ9XSxbMiwzLCIiLDAseyJsZXZlbCI6MSwic3R5bGUiOnsibmFtZSI6ImFkanVuY3Rpb24ifX1dXQ==
\[\begin{tikzcd}
	{\Coalg_Q^{C/-/C}} & {\Comod_C}
	\arrow[""{name=0, anchor=center, inner sep=0}, "{U^{\aug}_C}", shift left=2, from=1-1, to=1-2]
	\arrow[""{name=1, anchor=center, inner sep=0}, "{\Cofree^{\aug}_C}", shift left=2, from=1-2, to=1-1]
	\arrow["\dashv"{anchor=center, rotate=-90}, draw=none, from=0, to=1]
\end{tikzcd}\]
restricting the one from \ref{constr:U_aug}. Further this adjunction is comonadic.
\end{corollary}
\begin{proof}
    Since we know that $U^{\aug} \dashv \Cofree^{\aug}$ is a relative adjunction over $\Coalg_Q$, this follows from easily taking fibers at the coalgebra $C \in \Coalg_Q$. The comonadicity follows from \ref{lemma:fiber_limit}: we know that $q: \Coalg_Q^{\eqTriangle} \to \Coalg_Q$ preserves all limits. So does $p: \Comod_Q \to \Coalg_Q$. 

    Hence the fiber inclusions $\Coalg_Q^{C/-/C} \to \Coalg_Q^{\eqTriangle}$ and $\Comod_C \to \Comod_Q$ preserve totalizations. Since $U^{\aug}$ creates $U^{\aug}$-split totalizations (as it's comonadic), we see that $U^{\aug}_C$ must also preserve $U^{\aug}_C$-split totalizations! It's also conservative as $U^{\aug}$ is. Hence we're done, by the comonadicity theorem.
\end{proof}
%Add construction, and justification of comonadality

\subsection{Comodules and costability}\label{subsec:costcomod}
We now move to prove a central result about $Q$-cooperadic $C$-comodules for a $Q$-coalgebra $C$. Namely, we show that they are the costabilization of $C$-coaugmented $Q$-coalgebras. This is dual to the classical well known result that $P$-operadic $A$-modules are the stabilization of $A$-augmented $P$-algebras, for example see \cite{FrancisTangent}.

First of all, we fix a field of characteristic zero $k$. We will add to our assumptions that $\V \simeq \Mod_k$, thus $\C$ is a presentably stable $k$-linear category.

First we start with a calculation of pushouts of coalgebras, analogous to the fact that loops of algebras are trivial, for example see \cite[Example~2.4]{CalaqueGrivaux}.
\begin{proposition}\label{prop:coalg_pushout}
    Given a cooperad $Q$
    Let $D \in \Coalg_Q^{C/-/C}$. 
    Then $\Sigma_C D \simeq C \oplus (U_C^{\aug} D)[1]$.
\end{proposition}
\begin{proof}
    To show this, we first notice that it's enough to show this equivalence after forgetting the $C$-augmentation to $\Coalg_Q$, as the forgetful functor creates pushouts. 

    Now we analyze these $Q$-coalgebras as augmented cosimplicial objects. To do so, we denote coalgebras $B$ by 
    \[
    B \to Q^{\bullet} B
    \]
    (here we're identifying $B$ with its underlying $\C$-object).

    Since $D$ is $C$-augmented, its underlying object naturally splits as $C \oplus M$. Then we can take the augmented cosimplicial pushout of $C \leftarrow D \rightarrow C$, we we get a cosimplicial object 
    \[
    C \oplus M[1] \to \bigoplus_n Q(n) \otimes_{\Sigma_n} (C^n \oplus \Lin(M,n)[1] \oplus \Mult(M,n)[1]),
    \]
    where here $\Lin(M,n)$ denotes direct sum of the terms that have $n$ factors of $C, M$ but only a single $M$. On the other hand, $\Mult(M,n)$ denotes the direct sum of the terms $n$-factors of $C, M$ with at least two factors of $M$. This requires us to work over a field of characteristic zero. These are from the calculation of the pushout of $C^n \leftarrow D^n \rightarrow C^n$. We will denote this augmented cosimplicial object as $\Sigma^0_C D$.

    Notice that if we look at the augmented cosimplicial object of $\Sigma_C D$, we get 
    \[
    \Sigma_C D \to Q^{\bullet} \circ \Sigma_C D \simeq \bigoplus_n Q(n) \otimes_{\Sigma_n} (C^n \oplus \Lin(M,n)[1] \oplus \bigoplus^n_{b=2} \Mult(M,n,b)[b])
    \]
    where here $\Mult(M,n,b)$ is are the factors in $(C \oplus M)^n$ with exactly $b$ terms of $M$. 
    
    Notice that the natural comparison map from
    \[
    \Sigma^0_C D \to \Sigma_C D
    \]
    is given as follows: on the constant terms and linear terms 
    \[
    C^n \oplus \Lin(M,n)[1]
    \]
    it is an equivalence. On the higher degree terms in $M$, we have the natural $0$ map from suspension
    \[
    \Mult(M,n,b)[1] \to \Mult(M,n,b)[b].
    \]

    Notice that if we quotient by the $\Mult(M,n)[1]$ terms in $\Sigma^0_C D$, we get exactly the cosimplicial object corresponding to $C \oplus (U^{\aug}_C D[1])$. Indeed this quotient map exactly corresponds to the suspension of the colax morphism of \ref{constr:U_aug}! We're implicitly using also the fact that $C \oplus {-}: \Comod_C \to \Coalg_Q^{C/-/C}$ preserves suspensions, which one can see directly on the augmented cosimplicial objects.
    
    And so the natural map
    \[
    \Sigma^0_C D \to \Sigma_C D
    \]
    factors through this quotient
    \[
    \Sigma^0_C D \to \C \oplus (U^{\aug}_C D[1]) \to \Sigma_C D.
    \]

    However, clearly after taking the forgetful functor $U^{\aug}_C$, we see the morphism $C \oplus (U^{\aug}_C D[1]) \to \Sigma_C D$ gets sent to the identity of $\Sigma_C D \simeq C \oplus M[1]$. Hence since the forgetful functor is conservative, we see that the comparison map 
    \[
    \C \oplus (U^{\aug}_C D[1]) \to \Sigma_C D
    \] is an equivalence, as desired.
\end{proof}

Now we're ready to prove the theorem:
\begin{theorem}
    Fix a set $A$, $\V \in \CAlg(\Prlst)$, and $\C \in \CAlg(\PrlstV)$ following \ref{notation:AVC2}. Let $Q$ be a cooperad and $C$ be a $Q$-coalgebra. Then the functor 
    \[
    \oplus: \Comod_C \to \Coalg_Q^{C/-/C}
    \] defined in \ref{remark:aug_oplus} exhibits $\Comod_C$ as the costabilization of $\Coalg_Q^{C/-/C}$.
\end{theorem}

\begin{proof}
    We prove it as follows: the colimit-preserving functor \[
    \oplus: \Comod_C \to \Coalg_Q^{C/-/C}
    \]
    factors through the costabilization of $\Coalg_Q^{C/-/C}$:
    % https://q.uiver.app/#q=WzAsMyxbMCwwLCJcXENvbW9kX0MiXSxbMiwwLCJcXENvYWxnX1Fee0MvLS9DfSJdLFsxLDEsIlxcQ29TdCBcXENvYWxnX1Fee0MvLS9DfSJdLFswLDEsIlxcb3BsdXMiXSxbMCwyLCIiLDIseyJzdHlsZSI6eyJib2R5Ijp7Im5hbWUiOiJkb3R0ZWQifX19XSxbMiwxLCJcXFNpZ21hX3tcXGluZnR5fSIsMl1d
\[\begin{tikzcd}
	{\Comod_C} && {\Coalg_Q^{C/-/C}} \\
	& {\CoSt \Coalg_Q^{C/-/C}}
	\arrow["\oplus", from=1-1, to=1-3]
	\arrow["\widehat{\oplus}", dotted, from=1-1, to=2-2]
	\arrow["{\Sigma_{\infty}}"', from=2-2, to=1-3]
\end{tikzcd}\]
using the universal property of costabilization. Then there is a composite functor
\[
\beta: 
\begin{tikzcd}
    {\CoSt \Coalg_Q^{C/-/C}} \ar[r, "\Sigma_{\infty}"] & \Coalg_Q^{C/-/C} \ar[r, "U_C^{\aug}"] & \Comod_C
\end{tikzcd}
\]
which is excisive as both functors preserve colimits. We show that $\widehat{\oplus}$ and $\beta$ are inverse functors.

First let's notice that the composite $U_C^{\aug} \oplus$ is the identity of $\Comod_C$, which one can calculate on the underlying undecorated categories (ie on $\Fun(A \amalg dA, \C)$ and $\Fun(A, \C)^{\eqTriangle}$). 

Then notice that $U_C^{\aug} \oplus \simeq \beta \widehat{\oplus}$, by the following diagram:
% https://q.uiver.app/#q=WzAsNCxbMCwwLCJcXENvbW9kX0MiXSxbMiwwLCJcXENvYWxnX1Fee0MvLS9DfSJdLFsxLDEsIlxcQ29TdCBcXENvYWxnX1Fee0MvLS9DfSJdLFszLDAsIlxcQ29tb2RfQyJdLFswLDEsIlxcb3BsdXMiXSxbMCwyLCJcXHdpZGVoYXR7XFxvcGx1c30iLDIseyJzdHlsZSI6eyJib2R5Ijp7Im5hbWUiOiJkb3R0ZWQifX19XSxbMiwxLCJcXFNpZ21hX3tcXGluZnR5fSIsMl0sWzEsMywiVV57XFxhdWd9IiwyXV0=
\[\begin{tikzcd}
	{\Comod_C} && {\Coalg_Q^{C/-/C}} & {\Comod_C} \\
	& {\CoSt \Coalg_Q^{C/-/C}}.
	\arrow["\oplus", from=1-1, to=1-3]
	\arrow["{\widehat{\oplus}}"', dotted, from=1-1, to=2-2]
	\arrow["{\Sigma_{\infty}}"', from=2-2, to=1-3]
	\arrow["{U^{\aug}}"', from=1-3, to=1-4]
\end{tikzcd}\]

Hence we've shown that $\beta \widehat{\oplus} \simeq 1_{\Comod_C}$. 

Next we check the other composition. To do so, notice that it is enough to check that the following diagram commutes
% https://q.uiver.app/#q=WzAsNCxbMCwwLCJcXENvU3QgXFxDb2FsZ19RXntDLy0vQ30iXSxbMCwxLCJcXENvYWxnX1Fee0MvLS9DfSJdLFsxLDEsIlxcQ29tb2RfQyJdLFsyLDEsIlxcQ29hbGdfUV57Qy8tL0N9Il0sWzAsMSwiXFxTaWdtYV97XFxpbmZ0eX0iXSxbMSwyLCJVX0Nee1xcYXVnfSIsMl0sWzIsMywiXFxvcGx1cyIsMl0sWzAsMywiXFxTaWdtYV97XFxpbmZ0eX0iXV0=
\[\begin{tikzcd}
	{\CoSt \Coalg_Q^{C/-/C}} \\
	{\Coalg_Q^{C/-/C}} & {\Comod_C} & {\Coalg_Q^{C/-/C}}
	\arrow["{\Sigma_{\infty}}", from=1-1, to=2-1]
	\arrow["{U_C^{\aug}}"', from=2-1, to=2-2]
	\arrow["\oplus"', from=2-2, to=2-3]
	\arrow["{\Sigma_{\infty}}", from=1-1, to=2-3]
\end{tikzcd}\]
using the universal property of the costabilization. To show this, we use the proposition \ref{prop:coalg_pushout}.

Given an object of $\CoSt(\Coalg^{C/-/C}_Q)$ which we can think of as a sequence of $Q$-coalgebras $B^\bullet$ such that $\Sigma_C B^{n+1} \simeq B^{n}$. Then notice that $\Sigma_{\infty}$ sends this sequence $B^{\bullet}$ to $B^0$. 

Now we can calculate that since $B^0 \simeq \Sigma_C B^1$, we now that $B^0 \simeq C \oplus U^{\aug}_C B^0$, but this is exactly what we needed to show! In other words, we've shown that 
\[
\Sigma_{\infty} B^{\bullet} \simeq C \oplus U^{\aug}_C \Sigma_{\infty} B^{\bullet},
\]
so we're done.
\end{proof}

\DeclareFieldFormat{labelnumberwidth}{{#1\adddot\midsentence}}
\printbibliography
\end{document}